\documentclass[reqno,11pt]{amsart}
\usepackage{amsfonts, amsmath, amssymb, amsthm, color}
  \usepackage{amsfonts}
\usepackage{amssymb}
\usepackage{amsmath}
\usepackage{amsthm}
\usepackage{latexsym}
\usepackage{graphicx}
\usepackage{layout}
\usepackage{color}
\usepackage{geometry}
\usepackage{amsopn, enumitem}
\usepackage{palatino}

\allowdisplaybreaks[1]

\newtheorem{thm}{Theorem}[section]
\newtheorem{lemma}[thm]{Lemma}
\newtheorem{prop}[thm]{Proposition}

\newtheorem{rmk}[thm]{Remark}

\theoremstyle{definition}
 \numberwithin{equation}{section}

\newcommand{\cal}{\mathcal }
   
   \newcommand{\N}{\mathbb{N}}
\newcommand{\D}{\Delta}
\newcommand{\rr}{\mathbb{R}}
\newcommand{\R}{\mathbb{R}}
\newcommand{\al}{\alpha}
\newcommand{\de}{\delta}
 
\newcommand{\la}{\lambda}
\newcommand{\into}{\int_\Omega}
 \newcommand{\e}{\varepsilon}
 \renewcommand{\(}{\left(}
\renewcommand{\)}{\right)}
\renewcommand{\[}{\left[}
\renewcommand{\]}{\right]}
\newcommand{\beq}{\begin{equation}}
\newcommand{\eeq}{\end{equation}}
\def\bbm[#1]{\mbox{\boldmath $#1$}}
\begin{document}

\title[Asymmetric blow-up for the  $SU(3)$ Toda System]{Asymmetric blow-up for the  $\bbm[SU(3)]$ Toda System}
\author{Teresa D'Aprile}
\address[Teresa D'Aprile] {Dipartimento di Matematica, Universit\`a di Roma ``Tor
Vergata", via della Ricerca Scientifica 1, 00133 Roma, Italy.}
\email{daprile@mat.uniroma2.it}

\author{Angela Pistoia}
\address[Angela Pistoia] {Dipartimento SBAI, Universit\`{a} di Roma ``La Sapienza", via Antonio Scarpa 16, 00161 Roma, Italy}
\email{pistoia@dmmm.uniroma1.it}

\author{David Ruiz}
\address[David Ruiz]{Departamento de An\'{a}lisis Matem\'{a}tico, Granada, 18071 Spain.}
\email{daruiz@ugr.es}
\thanks{The authors have been supported by the Gruppo Nazionale per l'Analisi Matematica, la Probabilit\`a e le lore application (GNAMPA) of the Istituto Nazionale di Alta Matematica (IndAM)}
\thanks{The first and the second authors have been supported by the Italian PRIN Research Project 2012  \textit{Aspetti variazionali e perturbativi nei problemi differenziali nonlineari}.}
\thanks{The third author has been supported by the Spanish Ministry of
Science and Innovation under Grant MTM2011-26717 and by J.
Andalucia (FQM 116).}

\begin{abstract}
We consider the so-called Toda system in a smooth planar domain under homogeneous Dirichlet boundary conditions.
We prove the existence of a continuum of solutions for which both
components blow-up at the same point. This blow-up behavior is
asymmetric, and moreover one component includes also a
certain global mass. The proof uses singular perturbation methods.



\end{abstract}

\maketitle

\section{Introduction}

In this paper we consider the following version of the $SU(3)$ \textit{Toda system} on a smooth bounded domain $\Omega \subset \R^2$:

\begin{equation}\label{eq:e-1}
  \left\{
      \begin{aligned}&- \D u_1 = 2 \rho_1 \frac{e^{u_1}}{\int_{\Omega}e^{u_1}} - \rho_2 \frac{e^{u_2}}{\int_{\Omega}e^{u_2}}&  \hbox{ in }&  \Omega,\\
  &   - \D u_2 = 2 \rho_2 \frac{e^{u_2}}{\int_{\Omega}e^{u_2}} - \rho_1 \frac{e^{u_1}}{\int_{\Omega}e^{u_1}}&  \hbox{ in }&  \Omega, \\
    &  \ u_1=u_2=0 &  \hbox{ on }& \partial \Omega.
  \end{aligned}
    \right. \end{equation}
Here $\rho_1$, $\rho_2$ are positive constants. This problem, and its counterpart posed on compact surfaces of $\R^3$,  has been very much studied in the literature. The Toda system has a
close relationship with geometry, since it can be seen as the
Frenet frame of holomorphic curves in $\mathbb{CP}^N$ (see
\cite{guest}). Moreover, it arises in the study of the non-abelian
Chern-Simons theory in the self-dual case, when a scalar Higgs field
is coupled to a gauge potential, see \cite{dunne, tar, yys}.

Problem \eqref{eq:e-1} can also be seen as a natural generalization to systems   of the classical mean field equation. With respect to
the scalar case, the Toda system presents some analogies but also
some different aspects, which have attracted the attention of a lot
of mathematical research in recent years. Existence for the Toda
system has been studied from a variational point of view in
\cite{bjmr, cheikh, jw, mruiz}, whereas blowing-up solutions have
been considered in \cite{ao, lyan, lwzao-GAFA, mpwei, osuzuki},
for instance.

The blow-up analysis for the solutions to \eqref{eq:e-1}  was
performed in \cite{jlw}; let us explain it in some detail. Assume
that $u_n=({u_1}_n, {u_2}_n)$ is a blowing-up sequence of solutions of
\eqref{eq:e-1} with $({\rho_1}_n, {\rho_2}_n)$ bounded. Then, there
exists a finite blow-up set $S=\{p_1, \dots ,p_k\} \subset \Omega$
such that the solutions are bounded away from $S$. Concerning the
points $p_i$, let us define the local masses:

$$ \sigma_i = \lim_{r \to 0} \lim_{n \to +\infty} {\rho_i}_n \frac{\int_{B(p,r)} e^{{u_i}_n}}{\int_{\Omega} e^{{u_i}_n}}.$$

Then, the following scenarios are possible:

\begin{enumerate}
\item[a)] Partial blow-up: $(\sigma_1,\ \sigma_2)=(4\pi, 0)$ or $(\sigma_1,\ \sigma_2)=(0,4\pi)$. In such case, only one component is blowing up, and its profile is related to the entire solution of the Liouville problem in $\R^2$.

\item[b)] Asymmetric blow-up: $(\sigma_1,\ \sigma_2)=(4\pi, 8\pi) $ or $(\sigma_1,\ \sigma_2)=(8\pi, 4\pi) $. In this case, both components blow up and the local masses are different.

\item[c)] Full blow-up: $(\sigma_1, \ \sigma_2)=(8\pi, 8\pi)$. In this case, both components blow up and the local masses are equal.
\end{enumerate}

As a consequence of this study, the set of solutions is compact
for  any ${\rho} \in  (\R^+)^2 \setminus \mathcal{C}$, where

$$ \mathcal{C} = \left (4 \pi \N \times \R^+ \right ) \cup \left (\R^+ \times 4 \pi \N \right ).$$

See \cite{bmancini, jlw}. 
In other words, if blow-up occurs, at least one component ${u_i}_n$
is quantized, and ${\rho_i}_n \to 4 k \pi$ for some $k \in \N$.

Existence results of blowing-up solutions for the Toda system have
been found in \cite{ao, mpwei, lyan}, which concern partial
blow-up, asymmetric blow-up and full blow-up, respectively. In
those papers, $\rho_n$ converges to a single point of
$\mathcal{C}$.

Our starting point is the following observation: \emph{in the Toda
system one expects the existence of continua of families of
blowing-up solutions}. Indeed, if the Leray-Schauder degree of two
adjacent squares of $\R^2 \setminus \mathcal{C}$ is different,
then there must be blowing-up solutions for all points ${\rho}$ in
the common side.

In our preceding paper \cite{noi} we found continua of solutions
exhibiting partial blow-up. The same type
of solutions have been independently found in \cite{lin-wei-bo},
where the authors use them to compute the degree for the Toda
System when $\min \{ \rho_1, \ \rho_2\} < 8 \pi $.

In the present paper we
prove the existence of continua of solutions which develop
asymmetric blow-up. Indeed, given $\rho \in (4\pi, 8\pi)$, we are
able to find solutions for values $({\rho_1}_n, {\rho_2}_n) \to (8\pi,
\rho)$ or, analogously, $({\rho_1}_n, {\rho_2}_n)  \to (\rho, 8\pi)$
(see Figure 1).

\begin{figure}[h]
\centering
\includegraphics[width=0.5\linewidth]{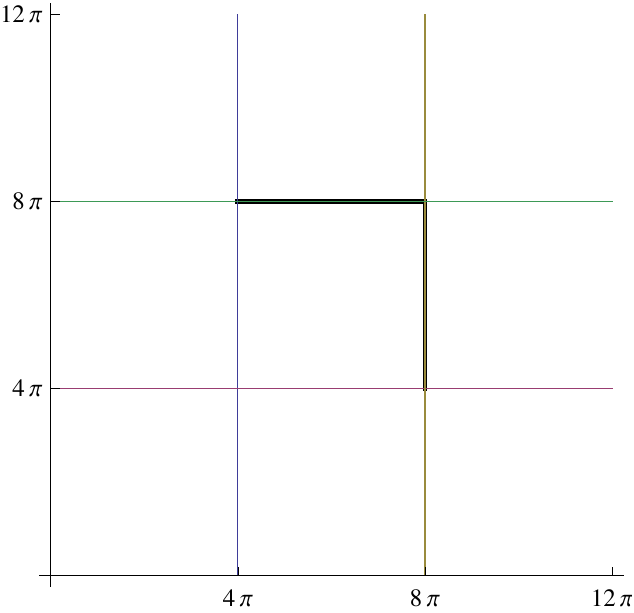}
\caption{ We find blowing-up solutions for which $(\rho_1, \rho_2)$
converges to any point of the two marked segments (excluding their endpoints).}
\end{figure}

We will assume throughout the paper that $\Omega$ is $k-$symmetric for some $k>2$
($k \in \N$), i.e.

\begin{equation}\label{ksym}
x\in\Omega\;\; \Longleftrightarrow\;\;   \Re_k(x)
\in\Omega,\quad \hbox{where}\quad
\Re_k(x):=\(\begin{matrix}\cos{2\pi\over k}&\sin{2\pi\over k}\\\\
-\sin{2\pi\over k}&\cos{2\pi\over k}\\
\end{matrix}\)\cdot x, \quad k>2.\end{equation}

\newpage
 In this paper we prove the following theorem.
\begin{thm} \label{teo} Let $\Omega$ be $k$-symmetric according to \eqref{ksym} and assume $0 \in \Omega$. Then, for
any $\rho \in (4\pi, 8\pi)$, there exists a family of
blowing-up solutions $({u_1}_\lambda, {u_2}_\lambda)$ of \eqref{eq:e-1} for $\lambda
\in (0, \lambda_0)$.

Such a  family has a unique blowing-up point at the
origin as $\la \to 0$, and $(\sigma_1, \sigma_2)=(4\pi, 8\pi)$
(asymmetric blow-up). Moreover, the corresponding values
$({\rho_{1}}_\lambda, {\rho_{2}}_\lambda)$ satisfy $${\rho_{1}}_\lambda=\rho,\qquad {\rho_{2}}_\lambda\to 8\pi.$$

\medskip
Concerning the asymptotic behavior of  the solutions, if we
make the change of variable \beq\label{change}u_1=2v_1-v_2,\;\; u_2= 2v_2-v_1,\eeq then
the following holds:

\begin{enumerate}
\item ${v_1}_\la(x)=\big (
-\log(\delta_1^2 + |x|^2) + 4 \pi H(x,0)\big ) + \frac 1 2 z(x) +
o(1)$ in $H^1(\Omega)$-sense, where\footnote{We use the notation $\sim$ to denote quantities which in the limit $\la\to 0^+$ are of the same order. }

\begin{equation} \label{difference2}
\delta_1=\delta_1(\la) \sim\sqrt{\lambda}\;\; \mbox{ as } \lambda \to 0;\end{equation}

\item ${v_2}_\la(x) = -\log(\delta_2^4 + |x|^4) + 8 \pi H(x,0) + o(1)$
in $H^1(\Omega)$-sense, where

\begin{equation} \label{difference}
\delta_2= \delta_2(\la) \sim \sqrt[4]{\lambda}\;\; \mbox{ as } \lambda \to 0.
\end{equation}

\end{enumerate}
Here  $H(x,y)$ denotes  the regular part of the Green's function and $z$ is the unique solution to the mean field equation
\begin{equation}\label{mf0}
\left\{\begin{aligned}&\Delta z+ 2(\rho-4\pi) \displaystyle{e^{z}\over\into e^{z}}=0 &\ \mbox{in} &\ \Omega, \\
& z =0 &\ \mbox{on} &\ \partial \Omega.
\end{aligned}
\right.
\end{equation}
\end{thm}

Let us give a couple of comments on the assumptions of Theorem \ref{teo}. For $\rho \in (4\pi, 8\pi)$, problem \eqref{mf0} admits a unique
solution which is also nondegenerate, as has been proved in
\cite{blin, suzuki}. Existence, uniqueness and nondegeneracy are
the reasons for the restriction $\rho <8 \pi$. Moreover, the symmetry requirement in Theorem \ref{teo} is used to rule out
the degeneracy of the radial solution of the singular Liouville
problem, see Proposition \ref{esposito} below.

Up to our knowledge, the only paper dealing with the existence of
asymmetric blow-up for the Toda system is \cite{mpwei}. In the construction of  \cite{mpwei}, $\rho_n
\to (4\pi, 8\pi)$, that is, there is no global mass. Our arguments
follow some of the ideas of that paper, but some interesting
differences have arose in our study. Observe that ${v_1}_\la$ contains a
peak around the origin which behaves as a solution of the regular
Liouville problem, suitably rescaled. Moreover, ${v_2}_\la$ is also
blowing-up at the origin at a lower speed. In a certain sense,
$e^{{v_1}_\la}$ acts as a Dirac delta for ${v_2}_\la$, and hence the limit
profile of ${v_2}_\la$ is the solution of a singular Liouville equation. Finally,
${v_1}_\la$ contains also a macroscopic part, $z(x)$, which yields the
global mass of the first component. This is one of the novelties
with respect to \cite{mpwei}. At this scale the concentration
effects of ${v_1}_\la$ and ${v_2}_\la$ cancel, and hence $z$ takes the form of
a solution of a regular Liouville problem posed in $\Omega$.

A second difference is that our two scales of concentration
(represented by the parameters $\delta_1$ and $\delta_2$) are
\emph{different from those in} \cite{mpwei}. This choice has been
forced by the presence of the global mass, and implies that
$\int_{\Omega} e^{{u_1}}$ remains bounded. This feature has another
interesting implication; if we define $\tilde{u}_1=u_1 + \log
\frac{\rho_1}{\int_{\Omega} e^{u_1}}$, $\tilde{u}_2 = u_2 + \log
\frac{\rho_1}{\int_{\Omega} e^{u_1}}$, we obtain solutions of the problem:
$$\left\{
    \begin{array}{lr}
      - \D \tilde{u}_1 = 2 e^{\tilde{u}_1}- e^{\tilde{u}_2} \ \ x \in \Omega,\\
     - \D \tilde{u}_2 = 2 e^{\tilde{u}_2}- e^{\tilde{u}_1}  \ \  x \in \Omega, \\
      \ \int_{\Omega} e^{\tilde{u}_i} < +\infty.
    \end{array}
    \right.$$
Those solutions are an example in which the singular set
for both components reduces to the origin but only the second
component diverges to $-\infty$ outside the origin. In other
words, the generalization of the classical Brezis-Merle result
\cite{breme} cannot involve both components in the Toda system.

The proofs use singular perturbation methods, which is based on the construction of suitable approximate solutions and on the study of
the invertibility of the linearized operator. This study is a third
difference with respect to \cite{mpwei}. Here, the first component
has a dual behavior, global and local, which implies an
interesting coupling between global and local terms, making the
whole proof more involved.

The rest of the paper in organized as follows. Section 2 is
devoted to some preliminary results, notation, and the definition
of our approximating solution. Moreover, a more general version of
Theorem \ref{teo} is stated there (see Theorem \ref{main}). The
error up to which the approximating solution solves our problem  is estimated in Section 3. In Section 4 we prove the solvability of the linearized problem. Finally, in Section 5, we
prove the existence result by a contraction mapping argument, and
we conclude the proof of Theorem \ref{teo} and Theorem \ref{main}.

\section{Preliminaries and statement of the main result}
In this section we will provide the \textit{ansatz} for solutions of problem \eqref{eq:e-1} and
 we will state our main result, which is a more general version of
Theorem \ref{teo}.

Motivated by the symmetry of the domain in assumption \eqref{ksym}, we consider symmetric functions, i.e., functions satisfying
\begin{equation}\label{even} u= u \circ \Re_k,\quad  \mbox{where $ \Re_k $ is defined in
\eqref{ksym}}.\end{equation}

We define:

$$ {\cal H}_k:=\left\{u\in {H}^1_0(\Omega):\ u \mbox{ satisfies  \eqref{even}} \right\}.$$

In order to construct our solutions, we will use the solution
$z$ to the problem:

\begin{equation}\label{mf}
\left\{\begin{aligned}
&\Delta z+ 2(\rho-4\pi) \displaystyle{e^{z}\over\into e^{z}}=0& \ \mbox{in }& \ \Omega, \\
& z =0& \ \mbox{on }& \ \partial \Omega.
\end{aligned}
\right.
\end{equation}

We shall need a nondegeneracy assumption on such
solution, in the following form:

\begin{enumerate}[label=(H), ref=(H)]
\item \label{H} Problem \eqref{mf} is solvable in ${\cal H}_k$ and the
solution (if not unique, at least one of them) is nondegenerate.
In other words, the linear problem
 $$\left\{\begin{aligned}
&\Delta \psi + 2(\rho-4\pi) \displaystyle{e^{z}\psi\over\into e^{z}dx}-2(\rho-4\pi) \displaystyle{e^{z}\into e^{z}\psi dx\over\(\into e^{z}dx\)^2}=0 \ & \mbox{in } &\ \Omega, \\
 &\psi =0 \ &\mbox{on } & \ \partial \Omega.
\end{aligned}
\right. $$
 admits only the trivial solution in the space ${\cal H}_k$.
\end{enumerate}

\begin{rmk} \label{ok} Problem \eqref{mf} always admits a solution if $\rho < 8 \pi$,
which is easily found as a minimizer of its corresponding energy
functional. Moreover, the solution is nondegenerate in this case,
even without symmetry restrictions (see \cite{suzuki} for the case
of a simply connected domain and \cite{blin} for the general
case). If $\rho \geq 8 \pi$ and $\Omega$ is the disk it is
well-known that there is no solution of \eqref{mf}. For a non
simply connected domain $\Omega$, instead, problem \eqref{mf}
admits a solution for all $\rho \neq 4 \pi (n+1)$, $n \in \N$, as
shown in \cite{clin} (see also \cite{djadli, dm} for a variational
approach). In this case, though, one expects nondegeneracy results
only for \emph{generic} domains $\Omega$.

\end{rmk}

In the rest of the paper we shall consider the following version
of the Toda system, with fixed $\rho\in (4\pi, 8\pi)$ and sufficiently small $\lambda>0$:
\begin{equation}
\label{s} \left\{\begin{aligned}
&\Delta u_1+ 2\rho{e^{u_1}\over\into e^{u_1}}-\lambda e^{u_2}=0& \quad \mbox{in } &  \Omega, \\
&\Delta u_2+ 2 \lambda e^{u_2}- \rho{e^{u_1}\over\into e^{u_1}}=0 &\quad \mbox{in }&  \Omega, \\
&u_1=u_2 =0 &\quad \mbox{on }&  \partial \Omega.
\end{aligned}
\right.
\end{equation}

 We now give a construction of a suitable approximate solution for \eqref{s}. To this aim, for $\al\ge2$, let us introduce the radially symmetric solutions of the singular Liouville problem $$
-\Delta w=|x|^{\al-2}e^w\quad \hbox{in}\;\; \rr^2,\qquad
\int_{\R^2} |x|^{\al-2}e^{w(x)}dx<+\infty.
$$
which are given by the one-parameter family of functions
$$
w^\al_\de(x):=\log 2\al^2{\de^\al\over\(\de^\al+|x|^\al\)^2}\quad
x\in\rr^2,\ \de>0.
$$
The following quantization property holds: \beq \label{quantum} \int_{\R^2}
|x|^{\al-2}e^{w_\de^\al(x)}dx = 4 \pi \alpha .\eeq
To obtain a better first approximation, we need to modify the functions $w_\de^\al$   in order to satisfy the zero boundary condition. Precisely, we consider the projections $P w_\de^\al $ onto the space $H^1_0(\Omega)$ of
$w_\de^\al$, where the projection  $P:H^1(\R^N)\to H^1_0(\Omega)$ is
defined as the unique solution of the problem
$$
 \Delta P u=\Delta u\quad \hbox{in}\ \Omega,\qquad  P u=0\quad \hbox{on}\ \partial\Omega.
$$

 We choose as initial approximation the following \textit{ansatz}:
$$W_\la=({W_1}_\la, {W_2}_\la),$$
\begin{equation}\label{answ}\begin{aligned} &
{W_1}_\la(x):= P w_1(x)-{1\over2} P w_2(x) +z(x),\\ & {W_2}_\la(x):=
P w_2(x)-{1\over2} P w_1(x) -{1\over2}z(x), \end{aligned}\end{equation}
where \begin{equation}\label{alfa}w_i(x):=w^{\al_i}_{\de_i}(x)\;\;\hbox{ with }
 \al_1:=2,\ \al_2:=4,
\end{equation}
and the values $\delta_i=\delta_i(\la)$ are defined as:

\begin{equation} \label{delta12} \delta_1 = \frac{1}{8} \sqrt{\frac{(\rho-4\pi) \lambda}{\int_{\Omega} e^{z} dx}} e^{6\pi
H(0,0)+\frac{z(0)}{4}},\qquad  \delta_2 = \frac{1}{2}\sqrt[4]\lambda
e^{3\pi H(0,0)- \frac{z(0)}{8}}. \end{equation} Here $H(x,y)$ denotes the regular part of the Green's function of $-\Delta$ over $\Omega$ under homogeneous Dirichlet boundary conditions, namely
$$H(x,y)=G(x,y)-\frac{1}{2\pi}\log\frac{1}{|x-y|}.$$

By the maximum principle we easily deduce the following asymptotic expansion
\beq\label{pro-exp}\begin{aligned}
 P  w_i(x)=& w_i(x)-\log\(2\al_i^2\de_i^{\al_i}\)+4\pi\al_i H(x,0)+O\(\de_i^{\al_i}\)\\ =&-2\log\(\de_i^{\al_i}+|x|^{\al_i}\)+4\pi\al_i H(x,0)+O\(\de_i^{\al_i}\)
\end{aligned}\eeq uniformly for $x\in \Omega$.

We shall look for a solution to \eqref{s} in a small neighbourhood of the first approximation, namely a solution of the form
 $$\({u_1}_\la, {u_2}_\la\)=W_\la+ {\phi}_\la,$$ where the rest term
$\phi_\la:=\({\phi_1}_\la, {\phi_2}_\la\)$ is small in
$H^1(\Omega)$-norm.

We are now in the position to state  the main theorem of the paper.

\begin{thm} \label{main} Let $\Omega$ be $k$-symmetric according to \eqref{ksym} and $0 \in \Omega$.
Assume that $\rho > 4 \pi$ and condition \ref{H} holds. Then, there
exists $\la_0>0$ such that for any $\la \in (0,\la_0)$ there
is $\phi_\la \in {\cal H}_k \times {\cal H}_k$ such that the couple $
({W_1}_\la+ {\phi_1}_\la, {W_2}_\la + {\phi_2}_\la)$ solves problem
\eqref{s}. Moreover, for any fixed $\e>0$,
$$\| \phi_\la \|_{(H^1_0(\Omega))^2}\leq \la^{\frac14- \e}\;\;\hbox{ for } \la \hbox{ sufficiently small}.$$
\end{thm}
As we shall see at the end of the paper,  Theorem \ref{teo} follows quite directly from Theorem \ref{main}.

We end up this section by setting the notation and basic well-known
facts which will be of use in the rest of the paper. We  denote by  $\|\cdot\|$ and $\|\cdot\|_p$  the norms in  the space $H^1_0(\Omega)$ and $L^p(\Omega)$, respectively, namely
 \beq\label{nott}\|u\|:=\|u\|_{H^1_0(\Omega)}
 ,\qquad \|u\|_p:=\|u\|_{L^p(\Omega)}
 \quad \forall u\in H^1_0(\Omega).\eeq
Moreover, if $u=(u_1,u_2)$, we denote:
$$\|u \|=\|u_1\|+\|u_2\|,\quad \|u\|_p=\|u_1\|_p+\|u_2\|_p. $$

In next lemma we recall the well-known Moser-Trudinger inequality
(\cite{Moe, Tru}).

\begin{lemma}\label{tmt} There exists $C>0$ such that for any bounded domain $\Omega$ in $\rr^2$
 $$\into e^{\frac{4\pi u^2}{\|u\|^2}}dx\le C |\Omega|\quad \forall u\in{ H}^1_0(\Omega),$$ where  $|\Omega|$ stands for the measure of the domain $\Omega$.
 In particular,  for any $q\geq 1$
 $$\| e^{u}\|_{q}\le  C^{\frac1q} |\Omega|^{\frac1q} e^{{q\over 16\pi}\|u\|^2}\quad \forall  u\in{H}^1_0(\Omega).$$

\end{lemma}

For any
$\alpha\ge2$ we will make use of the Hilbert spaces
\begin{equation}\label{ljs}
\mathrm{L}_\alpha (\rr^2):=\left\{u \in {\rm
W}^{1,2}_{loc}(\rr^2)\ :\  \left\|{|y|^{\frac{\alpha-2}{2}} \over
1+|y|^\alpha}u\right\|_{{L}^2(\rr^2)}<+\infty\right\}\end{equation}
 and
\begin{equation}\label{hjs}\mathrm{H}_\alpha (\rr^2):=\left\{u\in {\rm W}^{1,2}_{loc}(\rr^2) \ :\ \|\nabla u\|_{{L}^2(\rr^2)}+\left\|{|y|^{\frac{\alpha-2}{2}} \over 1+|y|^\alpha}u\right\|_{{L}^2(\rr^2)}<+\infty\right\},\end{equation}
 endowed with the norms
$$\|u\|_{\mathrm{L}_\alpha }:= \left\|{|y|^{\frac{\alpha-2}{2}} \over 1+|y|^\alpha}u\right\|_{{L}^2(\rr^2)}\
\hbox{and }\ \|u\|_{\mathrm{H}_\alpha }:= \(\|\nabla
u\|^2_{{L}^2(\rr^2)}+\left\|{|y|^{\frac{\alpha-2}{2}} \over
1+|y|^\alpha}u\right\|^2_{{L}^2(\rr^2)}\)^{1/2}.$$
We  denote by $\langle u,v \rangle_{\mathrm{L}_\alpha }$
the natural scalar product in ${\mathrm{L}_\alpha }$.

\begin{prop}\label{compact}
The embedding $i_\al:\mathrm{H}_\alpha
(\rr^2)\hookrightarrow\mathrm{L}_\alpha (\rr^2)$ is compact.
\end{prop}
\begin{proof}
  See \cite[Proposition 6.1]{gpistoia}.
\end{proof}

As commented in the introduction, our proof uses the singular
perturbation methods. For that, the nondegeneracy of the functions
that we use to build our approximating solution is essential. Next
proposition is devoted to the nondegeneracy of the entire
solutions of the Liouville equation (regular and singular).

\begin{prop}
\label{esposito} Assume that $\phi:\R^2\to\R$ satisfies \eqref{even} with $k>2$ and solves the equation
\begin{equation}\label{l1}
-\Delta \phi =2\alpha^2{|y|^{\alpha-2}\over (1+|y|^\alpha)^2}\phi\;\;
\hbox{in}\ \rr^2,\quad \int_{\R^2}|\nabla
\phi(y)|^2dy<+\infty,
\end{equation}
with $\alpha=2$ or $\alpha=4$. Then there exists $\gamma \in\rr$
such that
$$\phi(y)=\gamma   {1-|y|^\alpha\over 1+|y|^\alpha}.$$
\end{prop}
\begin{proof}
In \cite[Theorem 6.1]{gpistoia} it was proved that any solution
$\phi$ of \eqref{l1} is actually a bounded solution. Therefore we
can apply the result in \cite{dem} to conclude that $\phi= c_0 \phi_0 + c_1
\phi_1 + c_2 \phi_2$ for some $c_0,c_1,c_2\in \R$, where
$$\phi_0(y):   ={1-|y|^\alpha\over 1+|y|^\alpha} ,\ \; \;\phi_1(y):={ |y|^{  \frac\alpha2}\over 1+|y|^\alpha} \cos\frac{\alpha }2\theta
,\ \;\;\phi_2(y):={ |y|^{\frac\alpha2}\over 1+|y|^\alpha} \sin
\frac\alpha2 \theta.
$$ In the above definitions we have used polar coordinates.
Note that $\phi_0$ is radially symmetric and hence it satisfies
\eqref{even}; thus, $c_1 \phi_1 + c_2 \phi_2$ must satisfy
\eqref{even}. Observe now that $$c_1 \phi_1(y) + c_2 \phi_2(y)=A  \frac{|y|^{\frac{\alpha}{2}}}{1+|y|^\alpha} \sin\Big(\frac{\alpha}{2}\theta + \theta_0\Big)$$ with  $A=\sqrt{c_1^2 + c_2^2}$, $\theta_0\in \R$. Since $k>2$, we get that $A=0$ and, consequently,  $c_1=c_2=0$, concluding the proof.
\end{proof}

\begin{rmk}
The validity of Proposition \ref{esposito} is the main reason for the symmetry requirement
\eqref{ksym}.\end{rmk}

In our estimates throughout the paper, we will frequently denote by $C>0$, $c>0$ fixed
constants, that may change from line to line, but are always
independent of the variable under consideration. We also use the
notations $O(1)$, $o(1)$, $O(\lambda)$, $o(\lambda)$ to describe
the asymptotic behaviors of quantities in a standard way.

\section{Estimate of the error term}
The goal of this section is to  provide an estimate of the error up to which the couple $({W_1}_\la, {W_2}_\la)$ solves system \eqref{s}.
First of all, we perform the following estimate.
\begin{lemma}\label{aux} Define
$${E_1}_\la:=2\rho\frac{e^{{W_1}_\la}}{\into e^{{W_1}_\la}}-e^{w_1}-2(\rho-4\pi)\frac{e^{z}}{\into e^z},\quad {E_2}_\la:=2\la e^{{W_2}_\la}-|x|^2e^{w_2}.$$ For any $p\geq 1$ the following holds
$$\|{E_1}_\la\|_{p}=O(\la^{\frac{2-p}{2p}}),\,\;\; \|{E_2}_\la\|_{p}=O(\la^{\frac{2-p}{4p}}).$$
\end{lemma}

\begin{proof}
By \eqref{pro-exp} we compute

$$e^{{W_1}_\la}=e^{Pw_1-{1\over2}  Pw_2+z}=\frac{\delta_2^4+|x|^4}{(\delta_1^2+|x|^2)^2}e^{z+O(\delta_1^2)+O(\delta_2^4)}=\frac{\delta_2^4+|x|^4}{(\delta_1^2+|x|^2)^2}e^{z}(1+O(\la))$$ and, since $\frac{|x|^4}{(\delta_1^2+|x|^2)^2}=1+O(\frac{\delta_1^2}{\delta_1^2+|x|^2})$, we deduce
\beq\label{W1}e^{{W_1}_\la}=\frac{\delta_2^4}{(\delta_1^2+|x|^2)^2}e^{z}(1+O(\la))+e^{z}+O\Big(\frac{\la}{\la+|x|^2}\Big).\eeq
Then we scale the first term as $x=\delta_1 y$ and, owing to $e^{z(\delta_1 y)} =e^{z(0)}(1+O(\delta_1 y))$, we get
\beq\label{unoo}e^{{W_1}_\la(x)}=\frac{\delta_2^4}{\delta_1^4(1+|y|^2)^2}e^{z(0)}(1+O(\la)+O(\sqrt\la |y|))+e^{z(x)}+O\Big(\frac{\la}{\la+|x|^2}\Big).\eeq
Therefore, using that $\int_{\frac{\Omega}{\delta_1}}\frac{1}{(1+|y|^2)^2}=\int_{\R^2}\frac{dy}{(1+|y|^2)^2}+O(\sqrt\la)=\pi+O(\sqrt\la)$, we arrive at   \beq\label{due}\into e^{{W_1}_\la}dx=
\frac{\delta_2^4}{\delta_1^2}e^{z(0)}\pi +\into e^{z} dx +O(\sqrt\la).\eeq
In view of  the choice of $\delta_1, \delta_2$ in \eqref{delta12},  we have that \eqref{due} can be rewritten in the following two forms:
\beq\label{tre}\into e^{{W_1}_\la}dx=\frac{\delta_2^4}{\delta_1^2}\frac{\rho }{4}e^{z(0)}+O(\sqrt\la)=
\frac{\rho}{\rho-4\pi} \into e^z dx +O(\sqrt\la),\eeq and, combining   \eqref{tre} with \eqref{unoo},

$$\begin{aligned}2\rho\frac{e^{{W_1}_\la(x)}}{\into e^{{W_1}_\la}} &=\frac{8(1+O(\sqrt\la)+O(\sqrt\la |y|))}{\delta_1^2(1+|y|^2)^2}+2(\rho-4\pi)\frac{e^{z(x)}}{\into e^z}(1+O(\sqrt\la))+O\Big(\frac{\la}{\la+|x|^2}\Big)
\\ &= e^{w_1(x)}+O\Big(\frac{1}{\sqrt\la(1+|y|)^3}\Big)+2(\rho-4\pi)\frac{e^{z(x)}}{\into e^z}(1+O(\sqrt\la))+O\Big(\frac{\la}{\la+|x|^2}\Big).\end{aligned}$$
Therefore, it follows that
$$\Bigg\|2\rho\frac{e^{{W_1}_\la}}{\into e^{{W_1}_\la}} -e^{w_1}-2(\rho-4\pi)\frac{e^{z}}{\into e^z}\Bigg\|_{p}=O(\la^{\frac{2-p}{2p}}).$$
We turn our attention to the estimate of $e^{ {W_2}_\la}  $; by \eqref{pro-exp} we obtain

\beq\label{W2}e^{ {W_2}_\la}= e^{Pw_2-{1\over2} Pw_1-{1\over2} z } =
 \frac{\delta_1^2+|x|^2}{(\delta_2^4+|x|^4)^2} e^{12\pi H(x,0)-\frac{z}{2}}(1+O(\la)).\eeq
Now we scale $x=\delta_2 y$:
$$\begin{aligned}e^{ {W_2}_\la(x)}&=
 \frac{\delta_1^2+\delta_2^2|y|^2}{\delta_2^8(1+|y|^4)^2} e^{12\pi H(\delta_2 y,0)-\frac{z(\delta_2 y)}{2}}(1+O(\la))\\ &=
  \frac{|y|^2}{\delta_2^6(1+|y|^4)^2} e^{12\pi H(\delta_2 y,0)-\frac{z(\delta_2 y)}{2}}(1+O(\la))+O\Big(\frac{1}{\la (1+|y|^4)^2} \Big)
 \\ &=\frac{|y|^2}{\delta_2^6(1+|y|^4)^2} e^{12\pi H(0,0)-\frac{z(0)}{2}}\big(1+O(\la)+O(\sqrt[4]\la |y|)\big)+O\Big(\frac{1}{\la (1+|y|^4)^2} \Big)
.\end{aligned}$$
The choice of $\delta_2$ in \eqref{delta12} yields
$$\begin{aligned}2\la e^{ {W_2}_\la(x)}=|x|^2e^{w_2(x)}+O\Big(\frac{1}{\sqrt[4]\la(1+|y|^4)} \Big)\end{aligned}$$ and we conclude
$$\begin{aligned}\big\|2\la e^{ {W_2}_\la}-|x|^2e^{w_2}\|_{p}=
O\Big(\la^{\frac{2-p}{4p}}\Big).\end{aligned}$$
\end{proof}

Now we are going to  estimate the error term ${R}_\la$,
\begin{equation}\label{rla}
{R}_\la:=   \( {R_1} _\la,  {R_2} _\la \),\end{equation} where
$$\begin{aligned} &{R_1}_\la:=-\Delta {W_1}_\la-2\rho\frac{ e^{ {W_1}_\la}}{\into e^{{W_1}_\la}}  +\la  e^ {{W_2}_\la},\\ & {R_2} _\la:=-\Delta {W_2}_\la-2\la e^{ {W_2}_\la}  + \rho {e^{{W_1}_\la}\over\into e^{{W_1}_\la} }
 .
\end{aligned}$$

  \begin{lemma}\label{error} Let $ {R}_\la$ be as in \eqref{rla}.
Then for any $p\in [1,2] $  we have
$$\| {R}_\la\|_{p}= O\(\la^{{1\over4}{2-p\over p}}\).$$
\end{lemma}
\begin{proof} By \eqref{answ}, recalling that $z$ solves \eqref{mf},  we have
$$ \begin{aligned}
   {R_1} _\la &=
 -\Delta \(Pw_1-{1\over2}  Pw_2+z\)-2\rho {e^{{W_1}_\la}\over\into e^{{W_1}_\la}}+\la   e^{{W_2}_\la} \\
 &=\(e^{w_1}
- 2\rho\frac{e^{{W_1}_\la}}{\into e^{{W_1}_\la}}+2(\rho-4\pi)\frac{e^{z}}{\into e^z}
 \) -{1\over2}\( |x|^{2}e^{w_2}-2\la e^{{W_2}_\la} \). \end{aligned}$$
Analogously
$$ \begin{aligned}
   {R_2} _\la  &=\( |x|^{2}e^{w_2}-2\la e^{{W_2}_\la} \)-\frac12\(e^{w_1}
- 2\rho\frac{e^{{W_1}_\la}}{\into e^{{W_1}_\la} }+2(\rho-4\pi)\frac{e^{z}}{\into e^z}
 \) \end{aligned}$$
 and the thesis follows by applying Lemma \ref{aux}.
\end{proof}

\section{Analysis of the linearized operator}
Let us consider the following linear problem: given
$ h_1, h_2\in {\cal H}_k$,  find functions $\phi_1,\phi_2$ satisfying
\begin{equation}\label{lla}\left\{\begin{aligned}&
-\Delta \phi_1   +\la e^{{W_2}_\la}\phi_2   -2\rho \[{e^{{W_1}_\la}\phi_1\over \into e^{{W_1}_\la}dx}-
{e^{{W_1}_\la} \into e^{{W_1}_\la} \phi_1dx\over\( \into e^{{W_1}_\la}dx \)^2}  \]=\Delta h_1
 ,\\
 &-\Delta \phi_2 - 2\la e^{{W_2}_\la}\phi_2  +\rho \[{e^{{W_1}_\la}\phi_1\over \into e^{{W_1}_\la}dx}-
{e^{{W_1}_\la} \into e^{{W_1}_\la} \phi_1dx\over\( \into e^{{W_1}_\la}dx \)^2}  \] =\Delta h_2,\\ &
\phi_1,\phi_2\in {\cal H}_k.
\end{aligned}\right.
\end{equation}
\begin{prop}\label{inv}
For every $p\in (1,2)$ there exist $\lambda_0>0$ and $C>0$ such that for any $\la \in(0, \la_0)$,  any $h_1,h_2\in {\cal H}_k$ and  any
$\phi_1,\phi_2\in {\cal H}_k $ solutions of \eqref{lla}, the following holds $$\|\phi_1\|  +\|\phi_2\|  \leq C  |\log\la |  \Big(  \|h_1\| + \|h_2\|  \Big).$$
\end{prop}
\begin{proof}
We argue by contradiction. Assume that there exist $p\in(1,2),$ sequences
$\la_n\to0,$ ${h_i}_n\in{\cal H}_k$ and ${\phi_i}_n\in {\cal H}_k$ for $i=1,2$,
which solve \eqref{lla} and
\begin{equation}\label{inv2}
\|{\phi_1}_n\| +\|{\phi_2}_n\| =1, \end{equation}
\begin{equation}\label{inv3}
   |\log\la_n | (\|{h_1}_n\|+  \|{h_2}_n\|) \to 0.\end{equation}
We define  $\widetilde {\Omega_i}_n :={\Omega\over {\de_i}_n }$ and
 $$  \tilde {\phi_i}_n(y):=\left\{\begin{aligned}&{\phi_i}_n\({\de_i}_n y\)&\hbox{ if }&y\in \widetilde {\Omega_i}_n \\ &0&\hbox{ if }&y\in \rr^2\setminus\widetilde {\Omega_i}_n \end{aligned} \right. . $$
\\
 In what follows at many steps of the arguments we will pass to a subsequence, without further notice.  Moreover, for notational convenience, we avoid double subscripts and %the subscript `$M_n$' and
 we will simply write ${W_1}$,  ${W_2}$ in the place of ${W_1}_{\la_n}$, ${W_2}_{\la_n}$.

We split the remaining argument into five steps.
\medskip

\noindent{\em Step 1. We will show that
\begin{equation*}\label{step1.0} \tilde{\phi_1}_n \;\; \hbox{ is bounded in }\mathrm{H}_{2} (\rr^2),\end{equation*}
\begin{equation*}\label{step1.0.1} \tilde{\phi_2}_n\;\;   \hbox{ is bounded in }\mathrm{H}_{4} (\rr^2)
\end{equation*} (see  \eqref{hjs}).
  }

It is immediate to check that
\beq\label{bounabla}\int_{\R^2}|\nabla \tilde{\phi_i}_n|^2dy=\into|\nabla  {\phi_i}_n|^2dx\le1,\quad i=1,2.\eeq
Next, we multiply the first equation in \eqref{lla} by ${\phi_2}_n$, the second equation  by $2{\phi_2}_n$; then  we integrate over $\Omega$ and  sum up to obtain

$$\begin{aligned}3\la_n \into   e^{{W_2}} {\phi_2}_n^2dx=& 2\into |\nabla{\phi_2}_n|^2dx +\into \nabla{\phi_1}_n\nabla {\phi_2}_ndx+ \into \nabla {h_1}_n\nabla {\phi_2}_ndx\\ &+2\into \nabla {h_2}_n\nabla {\phi_2}_ndx\end{aligned}$$
which implies, by \eqref{inv2}--\eqref{inv3},
\beq\label{W22}\la_n \int_{\Omega}  e^{{W_2}}{\phi_2}_n^2dx \leq
C.\eeq So, Lemma \ref{aux}  gives $\into |x|^2
e^{w_2}{\phi_2}_n^2\leq C$ or,
equivalently,
$$\int_{\R^2} {|y|^2\over \(1+|y|^4\)^2} \tilde{\phi_2}_n^2  dy \le C.$$
Combining this with \eqref{bounabla}, we deduce that    $\tilde {\phi_2}_n$ is bounded in the space $\mathrm{H}_{4} (\rr^2)$.

We now consider the first component. First, let $\psi_1\in C^\infty_c(\Omega\setminus\{0\})$, $\psi_1\geq 0$ and
$\psi_1$ not identically zero. Then we multiply the first equation
in \eqref{lla} by ${\psi_1},$   we integrate over $\Omega$ and we
get
\begin{align}\label{1.1.1}&\int\limits_{{\Omega}} \nabla{\phi_1}_n\nabla\psi_1dx  - 2\rho \[{\into e^{{W_1}}{\phi_1}_n\psi_1dx \over \into e^{{W_1}}dx}-
{\into e^{{W_1}}\psi_1dx \into
e^{{W_1}} {\phi_1}_ndx\over\( \into e^{{W_1}}dx
\)^2}  \]\nonumber \\ &   +\la_n \into
e^{W_2}{\phi_2}_n\psi_1dx=-\into\nabla
{h_1}_n\nabla \psi_1dx.\end{align} We observe that by
\eqref{W1} and \eqref{W2}
\beq\label{unifo}e^{{W_1}}\to e^z , \;\;\;e^{W_2}\to \frac{1}{|x|^2}e^{12\pi H(x,0)-\frac{z}{2}} \hbox{ uniformly  on compact sets of } \Omega\setminus\{0\}\eeq and therefore, recalling \eqref{tre}, \eqref{1.1.1} yields
\beq\label{W11}\into e^{W_1}
{\phi_1}_ndx=O(1).\eeq

Now we multiply the first equation in \eqref{lla} by $2{\phi_1}_n$, the second equation by ${\phi_1}_n$,
we integrate over $\Omega$ and sum up to obtain
\beq\label{tosse}3\rho \[{\into e^{{W_1}}{\phi_1}_n^2dx\over \into e^{{W_1}}dx}-{
\Big( \into e^{{W_1}}{ \phi_1}_ndx\Big)^2\over\( \into e^{{W_1}}dx \)^2}  \]=O(1)
.\eeq

By taking into account \eqref{W11}, we get $\into
e^{W_1} {\phi_1}_n^2dx=O(1)$ which implies, by
Lemma \ref{aux},
$\into e^{w_1}{\phi_1}_n^2=O(1)$ or, equivalently,
$$\int_{\R^2}\frac{{\tilde{\phi_1}_n}^2}{(1+|y|^2)^2}dy=O(1)$$ and the thesis follows.
\bigskip

\noindent{\em Step 2. We will show that, for some $ \gamma_1, \gamma_2 \in\rr $,
\begin{equation}\label{step1.1} \tilde{\phi_1}_n\to \gamma_1  \frac{1-|y|^{2}}{1+|y|^{2}}\;\ \hbox{  weakly in $\mathrm{H}_{2} (\rr^2)$ and strongly in $\mathrm{L}_{2} (\rr^2) $, } \end{equation}
\begin{equation}\label{step1.2}  \tilde{\phi_2}_n \to \gamma_2  \frac{1-|y|^{4}}{1+|y|^{4}}\;\ \hbox{  weakly in $\mathrm{H}_{4} (\rr^2)$ and strongly in $\mathrm{L}_{4} (\rr^2) $,} \end{equation}
and
\begin{equation}\label{step1.3}  {\phi_1}_n\to 0 \;\ \hbox{weakly in $H^1_0(\Omega)$ and strongly in $L^q(\Omega)$ for any $q\ge2.$}\end{equation}
  }

Step 1 and Proposition \ref{compact} give
 $$\tilde {\phi_1}_n\to f\;\hbox{ weakly in }\mathrm{H}_2 (\rr^2) \hbox{ and strongly in }\mathrm{L} _2(\rr^2) .$$ Moreover, since $\|{\phi_1}_n\| \leq 1$,  $${\phi_1}_n\to g  \;\ \hbox{weakly in $H^1_0(\Omega)$ and strongly in $L^q(\Omega)$ for any $q\ge2.$}$$
Observe that each $\tilde {\phi_1}_n$ satisfies \eqref{even} owing to the definition of ${\cal H}_k$, then $f$ also satisfies \eqref{even}.
 Let $\tilde\psi_1\in C^\infty_c(\R^2)$ and set
 ${\psi_1}_n=\tilde{\psi_1}(\frac{x}{{\delta_1}_n})\in C^\infty_c(\Omega)$, for large $n$.
 We multiply the first equation in \eqref{lla} by ${\psi_1}_n,$   we integrate over $\Omega$ and we get
\begin{align}\label{1.1.1.1}&\int_{\widetilde{\Omega_1}_n} \nabla{\tilde{\phi_1}_n}\nabla\tilde\psi_1dy - 2\rho \[{\int_{{\Omega}} e^{W_1}{\phi_1}_n{\psi_1}_ndx \over \into e^{W_1}dx}-
{\into e^{W_1}{\psi_1}_ndx \int_{{\Omega}}
e^{W_1} {\phi_1}_ndx\over\( \into e^{W_1}dx
\)^2}  \]\nonumber \\ &   +\la_n \into
e^{W_2}{\phi_2}_n{\psi_1}_ndx=-\int_{{\Omega}}\nabla
{h_1}_n\nabla {\psi_1}_ndx.\end{align}
According to Lemma \ref{aux} we have
\beq\label{crucial}\begin{aligned}\frac{2\rho}{\into e^{W_1}}
\into e^{W_1}{\phi_1}_n dx&=\into
e^{w_1}{\phi_1}_n dx+2\frac{\rho-4\pi}{\into e^{z}} \into
e^{z} {\phi_1}_n+o(1)
\\ &=8\int_{\R^2} \frac{\tilde{\phi_1}_n }{(1+|y|^2)^2}dy+2\frac{\rho-4\pi}{\into e^{z}} \into e^{z} {\phi_1}_ndx+o(1)\\ &= 8\int_{\R^2} \frac{f }{(1+|y|^2)^2}dy+2\frac{\rho-4\pi}{\into e^{z}} \into e^{z} g\,dx+o(1).
\end{aligned}\eeq
Similarly,
$$\begin{aligned}\frac{2\rho}{\into e^{W_1}} \into e^{W_1}{\phi_1}_n{\psi_1}_n dx&
= 8\int_{\R^2} \frac{f \tilde{\psi}_1}{(1+|y|^2)^2}dy
+o(1)
\end{aligned}$$
and
$$\begin{aligned}\frac{2\rho}{\into e^{W_1}} \into e^{W_1}{\psi_1}_n dx&
= 8\int_{\R^2} \frac{\tilde{\psi}_1}{(1+|y|^2)^2}dy
+o(1)
\end{aligned}$$
Observe that $|x|^2e^{w_2}\leq C $ in the support of ${\psi_1}_n$; then, again by Lemma \ref{aux} we get $\la_ne^{W_2} {\psi_1}_n\to 0$ in $L^q(\Omega)$ for all $q\in (1,2)$, and so can estimate:

$$\begin{aligned}\la_n \int_{\Omega} e^{W_2}{\phi_2}_n{{\psi}_1}_n\, dx & = o(1)
. \end{aligned}$$
Finally, by \eqref{inv3}, using that $\into |\nabla {\psi_1}_n|^2=\int_{\R^2}|\nabla\tilde \psi_1|^2$,
\beq\label{macro}\int_{{\Omega}}|\nabla
{h_1}_n\nabla {\psi_1}_n|dx=O(\|{h_1}_n\|)=o(1).\eeq
\medskip Therefore, we may pass to the limit in \eqref{1.1.1.1} to
obtain
\begin{align}\nonumber&\int\limits_{\R^2} \nabla f \nabla\tilde\psi_1dy  =
8\int_{\R^2} \frac{f \tilde{\psi}_1}{(1+|y|^2)^2}dy\\ &\nonumber-\frac{4}{\rho}\int_{\R^2} \frac{\tilde{\psi}_1}{(1+|y|^2)^2}dy\bigg(8\int_{\R^2} \frac{f}{(1+|y|^2)^2}dy+2\frac{\rho-4\pi}{\into e^{z}dx} \into e^{z} g\,dx\bigg).\end{align}
Thus, since a single point has capacity zero in $\R^2$, we deduce that the function
$$f-\frac{4}{\rho}\int_{\R^2} \frac{f}{(1+|y|^2)^2}dy-\frac{\rho-4\pi}{\rho\into e^{z}} \into e^{z} g\,dx\in\mathrm{H}_2 (\rr^2)$$
is a solution of the equation
$$-\Delta\phi_0=\frac{8}{(1+|y|^2)^2}\phi_0\quad \hbox{ in }\R^2.$$
Recall now that $f$ satisfies \eqref{even} and,
by Proposition \ref{esposito}, \beq\label{crucial3}f-\frac{4}{\rho}\int_{\R^2} \frac{f}{(1+|y|^2)^2}dy-\frac{\rho-4\pi}{\rho\into e^{z}} \into e^{z} g\,dx=\gamma_1\frac{1-|y|^2}{1+|y|^2}\eeq
for some $\gamma_1\in\R$.
Since $\langle 1, \frac{1-|y|^2}{1+|y|^2}\rangle_{\mathrm{L}_2} =\int_{\R^2} \frac{1-|y|^2}{(1+|y|^2)^3}=0$, then $$\frac{4}{\rho}\int_{\R^2} \frac{f}{(1+|y|^2)^2}dy+\frac{\rho-4\pi}{\rho\into e^{z}} \into e^{z} g\,dx$$ actually coincides with the projection of $f$ onto the space of constants, namely, $$\frac{1}{\pi}\int_{\R^2}\frac{f}{(1+|y|^2)^2}dy=  \frac{4}{\rho}\int_{\R^2} \frac{f}{(1+|y|^2)^2}dy+\frac{\rho-4\pi}{\rho\into e^{z}} \into e^{z} g\,dx,$$
by which
\beq\label{crucial2}\frac{1}{\pi}\int_{\R^2}\frac{f}{(1+|y|^2)^2}dy=\frac{1}{\into e^{z}} \into e^{z} g\,dx,\eeq
By inserting this into \eqref{crucial} we get
\beq\label{crucial1}\frac{1}{\into e^{W_1}} \into e^{W_1}{\phi_1}_n dx=\frac{1}{\into e^{z} }\into e^{z} g\,dx+o(1).\eeq

Next let us fix $\psi_1\in C_c^\infty(\Omega\setminus\{0\})$; so,  we multiply the first equation in \eqref{lla} by ${\psi_1},$   we integrate over $\Omega$ and we get
$$\begin{aligned}&\into \nabla{\phi_1}_n\nabla\psi_1dx  - 2\rho \[{\into  e^{W_1}{\phi_1}_n\psi_1dx \over \into  e^{W_1}}-
{\into e^{W_1}\psi_1dx \into e^{W_1} {\phi_1}_ndx\over\( \into e^{W_1} \)^2}  \]\\ &+\la_n\into e^{W_2}{\phi_2}_n\psi_1dx=-\into\nabla {h_1}_n\nabla \psi_1.\end{aligned}$$
By \eqref{unifo}, and recalling \eqref{tre} and \eqref{crucial1},
we pass to the limit  to obtain
$$\into \nabla g\nabla\psi_1dx  - 2(\rho-4\pi) \[{\into  e^{z}g\psi_1dx \over \into  e^{z}}-
{\into e^{z}\psi_1dx \into e^{z} g\,dx\over\( \into e^{z} \)^2}  \]=0.$$
Since a single point has capacity zero in $\R^2$, then the above identity turns out to hold for every $\psi_1\in C^\infty_c(\Omega)$; we deduce that  $g\in H^1_0(\Omega)$ solves the problem
$$ -\Delta g  - 2(\rho-4\pi)\[ {  e^{z}g  \over \into  e^{z}}-
{   e^{z}  \into e^{z} g\, dx\over\(
\into e^{z} \)^2}  \]=0\ \;\;\hbox{in}\ \Omega.$$
Therefore, by \ref{H} we get $g=0$. \eqref{crucial2} gives
$\int_{\R^2}\frac{f}{(1+|y|^2)^2}=0$ and, finally, by
\eqref{crucial3} we conclude $f=\gamma_1\frac{1-|y|^2}{1+|y|^2}$.
We have thus proved \eqref{step1.1} and \eqref{step1.2}.

Observe that, in particular,

\begin{equation} \label{commento1} \int_{\Omega} e^{w_1}
{\phi_1}_ndx= 8\langle \tilde{\phi_1}_n, 1 \rangle_{\mathrm{L}_2} =o(1),\quad \ \int_{\Omega} e^{z} {\phi_{1}}_ndx=o(1).
\end{equation}

 In order to prove \eqref{step1.3}, observe that Step 1 and Proposition \ref{compact} give
 $$\tilde {\phi_2}_n\to h\;\hbox{ weakly in }\mathrm{H}_4 (\rr^2) \hbox{ and strongly in }\mathrm{L}_4(\rr^2) .$$
Since $\tilde {\phi_2}_n$ satisfies \eqref{even}, then $h$ also satisfies \eqref{even}.
Let us fix $\tilde\psi_2\in C_c^\infty(\R^2\setminus\{0\})$ and
set ${\psi_2}_n=\tilde{\psi_2}(\frac{x}{{\delta_2}_n})$,
$x\in\Omega$. Then ${\psi_2}_n\in
C^\infty_c(\Omega\setminus\{0\})$; so,  we multiply the second
equation in \eqref{lla} by ${\psi_2}_n,$   we integrate over
$\Omega$ and we get
\begin{align}\label{1.1.1.1.1}&\int\limits_{\widetilde{\Omega_2}_n} \nabla{\tilde{\phi_2}_n}\nabla\tilde\psi_2dy
+\rho \[{\int_{{\Omega}} e^{W_1}{\phi_1}_n{\psi_2}_ndx \over \into e^{W_1}dx}-
{\into e^{W_1}{\psi_2}_ndx \int_{{\Omega}} e^{W_1} {\phi_1}_ndx\over\( \into e^{W_1}dx \)^2}  \]\nonumber \\ &  -2\la_n\into e^{W_2}{\phi_2}_n{\psi_2}_n dx  =-\int_{{\Omega}}\nabla  {h_2}_n\nabla {\psi_2}_ndx\end{align}
Since $e^{w_1}\leq C$ in the support of ${\psi_2}_n$, then Lemma \ref{aux} gives $\frac{1}{\into e^{W_1}}e^{W_1}{\psi_2}_n\to 0$ in $L^q(\Omega)$ for all $q\in (1,2)$, and so
$$\frac{1}{\into e^{W_1}}\into e^{W_1}{\psi_2}_ndx=o(1), \quad \frac{1}{\into e^{W_1}}\into e^{W_1}{\phi_1}_n{\psi_2}_ndx=o(1).$$
Again by Lemma \ref{aux} we compute
$$2\la_n\into e^{W_2}{\phi_2}_n{\psi_2}_n dx =\into |x|^2e^{w_2}{\phi_2}_n{\psi_2}_n dx +o(1)=32\int_{\R^2}\frac{|y|^4}{(1+|y|^4)^2} h\tilde{\psi_2}dy +o(1).$$ Similarly to \eqref{macro}, we have $\int_{{\Omega}}\nabla  {h_2}_n\nabla {\psi_2}_n=o(1)$.
Then, recalling \eqref{crucial1} we can pass to the limit in \eqref{1.1.1.1.1} to obtain
\begin{align} \int_{\R^2} \nabla h\nabla\tilde\psi_2dy
-32 \int_{\R^2}\frac{|y|^4}{(1+|y|^4)^2}h  \tilde\psi_2
dy=0.\end{align}
The above identity turns out to hold for every
$\tilde\psi_2\in C^\infty_c(\R^2)$; we deduce that  $h\in {\mathrm
H}_4(\R^2)$ solves the problem
$$ -\Delta h  - 32\frac{|y|^4}{(1+|y|^4)^2}h=0\ \hbox{in}\ \Omega.$$
Recalling that $h$ satisfies \eqref{even}, Proposition \ref{esposito} implies $h=0$. This proves \eqref{step1.3}.

We point out that

\begin{equation} \label{commento2}  \int_{\Omega} |x|^2e^{w_2}
{\phi_{2}}_n dx= 32\langle \tilde{\phi_2}_n, 1 \rangle_{\mathrm{L}_4} =o(1).
\end{equation}

\bigskip

\noindent{\em Step 3. We will show that
$$\langle {\tilde{\phi_1}_n}, 1\rangle_{{\mathrm{L}}_2}-\frac{\pi}{\into e^z}\into e^{z} {\phi_1}_ndx=o\Big(\frac{1}{|\log\la_n|}\Big),$$
 $$\langle \tilde{\phi_2}_n, 1\rangle_{{\mathrm{L}}_4}=o\Big(\frac{1}{|\log\la_n|}|\Big).$$
  }
It is important to notice that, by \eqref{commento1} and
\eqref{commento2}, both expressions converge to $0$. This step is
devoted to prove an estimate on the speed of convergence, which
will be crucial for step 4.

Let   $Z_1   (x)={\delta_1^2-|x |^2\over \delta_1^2+|x |^2} $ $Z_2
(x)={\delta_2^4-|x |^4\over \delta_2^4+|x |^4} $ be the radial
solution to the linear problems
 $$-\Delta Z_1=8{\delta_1^2\over\(\delta_1^2+|x |^2\)^2}Z_1\;\; \hbox{in}\ \rr^2,\qquad \quad -\Delta Z_2=32{\delta_2^4|x|^2\over\(\delta_2^4+|x |^4\)^2}Z_2\ \;\hbox{in}\ \rr^2.$$
 Let $PZ_1$, $PZ_2$ be their projection   onto $H^1_0(\Omega)$. By the maximum principle it is not difficult to check that
\begin{equation}\label{pz0}\begin{aligned}
&PZ_1   =  Z_1 +1+O\(\de_1^2\)={ 2 \de_1^2 \over \de_1^2+|x |^{2} }+O\(\la\),
\\ &
PZ_2   =  Z_2 +1+O\(\de_2^4\)={ 2 \de_2^4 \over \de_2^4+|x |^{4} }+O\(\la\).
\end{aligned}\end{equation}
We observe that \beq\label{pizeta}\|PZ_1\|_q^q
=O\big({{\delta}_1^2}\big),\qquad
\|PZ_2\|_q^q=O\big({{\delta}_2^2}\big) \quad \forall
q>1 .\eeq
Now, we multiply the first equation in \eqref{lla} by $   PZ_1  $,
we integrate over $\Omega$ and we get
\begin{align}\label{1.2}& \into \nabla{\phi_1}_n\nabla PZ_1dx  -2\rho \[{\into e^{W_1}{\phi_1}_nPZ_1dx \over \into e^{W_1}dx}-
{\into e^{W_1}PZ_1dx \into e^{W_1}
{\phi_1}_ndx\over\( \into e^{W_1}dx \)^2}
\]\nonumber \\ &  +\la_n \into e^{W_2}{\phi_2}_n
PZ_1dx  =-\into \nabla {h_1}_n\nabla
PZ_1dx.\end{align}
We are now concerned with the estimates of each term of the above
expression.

Since, for any $q\geq 1$, \beq\label{sig1}\begin{aligned} \||x|^2
e^{w_2} PZ_1\|_q&=\Big\||x|^2 e^{w_2} { 2 {\de_1}_n^2
\over {\de_1}_n^2+|x |^{2} }\Big\|_{q}+O(\la_n)\||x|^2
e^{w_2}\|_{q}\\ &\leq 2 {\de_1}_n^2\big\|e^{w_2}
\big\|_{q}+O(\la_n)\||x|^2 e^{w_2}\|_{q} =
O(\la_n^{\frac{1}{2q}})
\end{aligned}\eeq
by Lemma \ref{aux} we conclude:

\beq\label{sigma3}\begin{aligned}2\la_n \into e^{W_2}{\phi_2}_n
PZ_1dx&=\into |x|^2e^{w_2}  {\phi_2}_n
PZ_1dx+o\Big(\frac{1}{|\log\la_n|}\Big)=o\Big(\frac{1}{|\log\la_n|}\Big).\end{aligned}\eeq

Next, we compute
\beq\label{sigma1}\begin{aligned}\into
\nabla{\phi_1}_n\nabla PZ_1dx=8\into
{\phi_1}_n{{\delta_1}_n^2\over\({\delta_1}_n^2+|x |^2\)^2}Z_1dx=\into
e^{w_1}{\phi_1}_nZ_1dx.
\end{aligned}\eeq
According to Step 1 we have \beq\label{simi} \begin{aligned} \into
|e^{w_1} {\phi_1}_n|dx=8\langle |\tilde{\phi_1}_n|,
1\rangle_{{\mathrm{L}}_2}=O(1). \end{aligned}\eeq

By Lemma \ref{aux}, using \eqref{pizeta} and \eqref{simi},
\beq\label{sigma2}\begin{aligned}\frac{2\rho}{\into
e^{W_1}}\into e^{W_1}{\phi_1}_nPZ_1dx&= \into
e^{w_1}{\phi_1}_nPZ_1dx+2\frac{\rho-4\pi}{\into e^z}\into
e^{z}{\phi_1}_nPZ_1dx+o\Big(\frac{1}{|\log\la_n|}\Big)
\\ &=\into e^{w_1}{\phi_1}_n(Z_1+1)dx
+o\Big(\frac{1}{|\log\la_n|}\Big).
\\ &
=\into e^{w_1}{\phi_1}_nZ_1dx+8\langle \tilde{\phi_1}_n, 1\rangle_{\mathrm{L}_2}
+o\Big(\frac{1}{|\log\la_n|}\Big).
\end{aligned}
\eeq Similarly, by using again Lemma \ref{aux} and \eqref{pizeta}
\beq\label{sigma4}\begin{aligned}\frac{2\rho}{\into e^{W_1}}\into
e^{W_1}PZ_1dx&=\into e^{w_1} (Z_1+1)dx
+o\Big(\frac{1}{|\log\la_n|}\Big)
\\ &
=\into e^{w_1}\frac{2{\delta_1}_n^2}{{\delta_1}_n^2+|x|^2}dx
+o\Big(\frac{1}{|\log\la_n|}\Big)
=8\pi +o\Big(\frac{1}{|\log\la_n|}\Big)
\end{aligned}\eeq
since $$\into
e^{w_1}\frac{{\delta_1}_n^2}{{\delta_1}_n^2+|x|^2}dx=8\int_{{\widetilde{\Omega_1}_n}}\frac{dy}{(1+|y|^2)^3}=8\int_{\R^2}\frac{dy}{(1+|y|^2)^3}+O({\delta_1}_n^2)=4\pi+O({\delta_1}_n^2).$$
Moreover,  by Lemma \ref{aux},

\beq \label{sigma6}\begin{aligned}\frac{2\rho}{\into e^{W_1}}\into
e^{W_1}{\phi_1}_n dx&= \into e^{w_1} {\phi_1}_n
dx+2\frac{\rho-4\pi}{\into e^z}\into e^{z}
{\phi_1}_ndx+o\Big(\frac{1}{|\log\la_n|}\Big)\\ &=8 \langle
\tilde{\phi}_{1\, n},1\rangle_{\mathrm{L}_2}
+2\frac{\rho-4\pi}{\into e^z}\into e^{z}
{\phi_1}_ndx+o\Big(\frac{1}{|\log\la_n|}\Big).\end{aligned}\eeq
Finally, since $PZ_1=O(1)$, we have  $\into |\nabla PZ_1|^2=\into
e^{w_1} PZ_1=O(\into e^{w_1})=O(1)$, by which, owing to
\eqref{inv3}, \beq\label{sigma5}\begin{aligned}\into
|\nabla {h_1}_n\nabla
PZ_1|dx&\leq\|{h_1}_n\| \|PZ_1\| =\|{h_1}_n\|=o\Big(\frac{1}{|\log\la_n|}\Big).\end{aligned}\eeq

We now multiply \eqref{1.2} by $\log\la_n$ and pass to the limit,
inserting \eqref{sigma3}, \eqref{sigma1}, \eqref{sigma2},
\eqref{sigma4}, \eqref{sigma6} and \eqref{sigma5}, to obtain

\beq\label{firstsigma}\log\la_n\langle \tilde{\phi_1}_n,
1\rangle_{{\mathrm{L}}_2}-\log\la_n\frac{\pi}{\into e^z}\into
e^{z} {\phi_1}_ndx=o(1)\eeq and the first part of the thesis
follows.

\bigskip

For the second part, we multiply the second equation in
\eqref{lla} by $PZ_2$ and integrate, to obtain:

\begin{align}\label{1.2bis}& \into \nabla{\phi_2}_n\nabla PZ_2dx  +\rho \[{\into e^{W_1}{\phi_1}_nPZ_2dx \over \into e^{W_1}dx}-
{\into e^{W_1}PZ_2dx \into e^{W_1}
{\phi_1}_ndx\over\( \into e^{W_1}dx \)^2}
\]\nonumber \\ &  -2\la_n \into e^{W_2}{\phi_2}_n
PZ_2dx  =-\into \nabla {h_2}_n\nabla
PZ_2dx.\end{align}

We now estimate each of the terms above. Observe that:

\beq\label{sigma11}\begin{aligned}\into
\nabla{\phi_2}_n\nabla PZ_2dx=32\into
{\phi_2}_n{{\delta_2}_n^4|x|^2\over\({\delta_2}_n^4+|x |^4\)^2}Z_2dx=\into
|x|^2e^{w_2}{\phi_2}_nZ_2dx .
\end{aligned} \eeq

According to Step 1, \beq\label{simi2}\into ||x|^2e^{w_2}
{\phi_2}_n|=32\langle |\tilde{\phi_2}_n|,
1\rangle_{{\mathrm{L}}_4}=O(1). \eeq

By Lemma \ref{aux}, \eqref{pizeta}  and \eqref{simi2} we compute
\beq\label{sigma22}\begin{aligned}2\la_n \into e^{W_2}{\phi_2}_n PZ_2dx &=\into|x|^2e^{w_2}{\phi_2}_n PZ_2dx  +o\Big(\frac{1}{|\log\la_n|}\Big)\\ &=
\into|x|^2e^{w_2}{\phi_2}_n (Z_2+1)dx  +o\Big(\frac{1}{|\log\la_n|}\Big)\\ &=\into|x|^2e^{w_2}{\phi_2}_n Z_2dx+32 \langle \tilde{\phi_2}_n, 1\rangle_{{\mathrm{L}}_4}+o\Big(\frac{1}{|\log\la_n|}\Big).
\end{aligned}\eeq
Next, by using Lemma \ref{aux}, recalling \eqref{pizeta} and
\eqref{simi}, \beq\label{sigma33}\begin{aligned}
\frac{2\rho}{\into e^{W_1}}\into e^{W_1}{\phi_1}_nPZ_2dx&=\into
e^{w_1}{\phi_1}_nPZ_2dx+2\frac{\rho-4\pi}{\into e^z}\into
e^{z}PZ_2 dx + o\Big(\frac{1}{|\log\la_n|}\Big)\\ &=\into
e^{w_1}{\phi_1}_n \frac{2{\delta_2}_n^4}{{\delta_2}_n^4+|x|^4}
dx+o\Big(\frac{1}{|\log\la_n|}\Big).
\end{aligned}\eeq
Let us observe that, using Step 1, and denoting by $\chi_A$ the characteristic function of the set $A$,  $$\begin{aligned}\int_{|x|\geq {\delta_2}_n^{3/2}}e^{w_1}|{\phi_1}_n|dx&= 8\langle \chi_{\big\{|y|\geq \frac{{\delta_2}_n^{3/2}}{{\delta_1}_n}\big\}}, \tilde{\phi_1}_n\rangle_{{\mathrm{L}_2}}\leq C \|\chi_{\big\{|y|\geq \frac{{\delta_2}_n^{3/2}}{{\delta_1}_n}\big\}}\|_{{\mathrm{L}_2}}
\\ &= C \Big(\int_{|y|\geq \frac{{\delta_2}_n^{3/2}}{{\delta_1}_n}}\frac{1}{(1+|y|^2)^2}dy\Big)^{1/2}\leq C\frac{{\delta_1}_n^{1/2}}{{\delta_2}_n^{3/4}}\end{aligned}$$
by which \eqref{sigma33} becomes
$$\begin{aligned}\frac{2\rho}{\into e^{W_1}}\into e^{W_1}{\phi_1}_nPZ_2dx&= \int_{|x|\leq {\delta_2}_n^{3/2}}e^{w_1}{\phi_1}_n\frac{2{\delta_2}_n^4}{{\delta_2}_n^4+|x|^4} dx  +o\Big(\frac{1}{|\log\la_n|}\Big)
\end{aligned}$$
and, since $ \frac{2{\delta_2}_n^4}{{\delta_2}_n^4+|x|^4}=2+O({\delta_2}_n^{2})$ for $|x|\leq {\delta_2}_n^{3/2}$,
\beq\label{sigma333}\begin{aligned}\frac{2\rho}{\into e^{W_1}}\into e^{W_1}{\phi_1}_nPZ_2dx&=2 \int_{|x|\leq {\delta_2}_n^{3/2}}e^{w_1}{\phi_1}_n dx  +o\Big(\frac{1}{|\log\la_n|}\Big)\\ &=2 \int_{\R^2}e^{w_1}{\phi_1}_n dx+o\Big(\frac{1}{|\log\la_n|}\Big)\\ &=16 \langle \tilde{\phi_1}_n, 1\rangle_{{\mathrm{L}}_2}+o\Big(\frac{1}{|\log\la_n|}\Big) .
\end{aligned}\eeq
The identical computation hold by replacing ${\phi_1}_n$ by $1$ in
\eqref{sigma333} and so
\beq\label{sigma44}\begin{aligned}\frac{2\rho}{\into e^{W_1}}\into
e^{W_1}PZ_2dx&
=16\|1\|_{{\mathrm{L}}_2}^2+o\Big(\frac{1}{|\log\la_n|}\Big)=16\pi
+o\Big(\frac{1}{|\log\la_n|}\Big).\end{aligned}\eeq Finally,
similarly to \eqref{sigma4},
\beq\label{sigma44bis}\begin{aligned}\into \nabla
{h_1}_n\nabla
PZ_1dx&=o\Big(\frac{1}{|\log\la_n|}\Big)\end{aligned}\eeq

By multiplying \eqref{1.2bis} by $\log\la_n$, passing
to the limit, and inserting \eqref{sigma11}, \eqref{sigma22},
\eqref{sigma333}, \eqref{sigma44} and \eqref{sigma44bis}, and
recalling \eqref{sigma6}, we arrive at

\begin{equation}\label{secondsigma}-\log\la_n\langle \tilde{\phi_2}_n, 1\rangle_{\mathrm{L}_4}+\frac{\rho-4\pi}{4\rho}\log\la_n
\Big(\langle \tilde{\phi_1}_n, 1\rangle_{\mathrm{L}_2}
-\frac{\pi}{\into e^{z} }\into e^{z}{\phi_1}_n dx \Big)=o(1) .\eeq
Combining \eqref{firstsigma} with \eqref{secondsigma} we obtain the thesis.

\bigskip

Before going on, we recall the following identities which follow by straightforward computations: for every $\alpha\geq 2$

\beq\label{comput1}\Big\langle
\frac{1-|y|^{\alpha}}{1+|y|^{\alpha}}, 1\Big\rangle_{{\mathrm
L}_\alpha}=\int_{\R^2}\frac{|y|^{\alpha-2}}{(1+|y|^\alpha)^2}\frac{1-|y|^\alpha}{1+|y|^\alpha}dy=0,\eeq

\beq\label{comput2}\Big\langle
\frac{1-|y|^{\alpha}}{1+|y|^{\alpha}},
\log(1+|y|^\alpha)\Big\rangle_{{\mathrm
L}_\alpha}=\int_{\R^2}\frac{|y|^{\alpha-2}}{(1+|y|^\alpha)^2}\frac{1-|y|^\alpha}{1+|y|^\alpha}\log(1+|y|^\alpha)dy=-\frac{\pi}{\alpha},\eeq
\beq\label{comput3}\Big\langle
\frac{1-|y|^{\alpha}}{1+|y|^{\alpha}},
\log|y|\Big\rangle_{{\mathrm
L}_\alpha}=\int_{\R^2}\frac{|y|^{\alpha-2}}{(1+|y|^\alpha)^2}\frac{1-|y|^\alpha}{1+|y|^\alpha}\log
|y|dy=-\frac{\pi}{2\alpha^2}.\eeq

\bigskip

\noindent\textit{Step 4. We will show that $\gamma_1=\gamma_2=0$.}

\bigskip

We first multiply the first equation in \eqref{lla} by $Pw_1$ and
integrate;
we obtain:

\beq\label{1.3}\begin{aligned}& \into \nabla{\phi_1}_n\nabla Pw_1dx -2\rho \[{\into e^{W_1}{\phi_1}_nPw_1dx \over \into e^{W_1}dx}-
{\into e^{W_1}Pw_1dx  \into
e^{W_1} {\phi_1}_ndx\over\( \into e^{W_1}dx
\)^2}  \] \\ & + \la_n \into
e^{W_2}{\phi_2}_n Pw_1dx     =-\into
\nabla{h_1}_n\nabla Pw_1dx.\end{aligned}\eeq

Let us estimate each of the terms above. By \eqref{commento1},

\beq\label{ludo1}\into \nabla{\phi_1}_n\nabla
Pw_1dx=\into  e^{w_1} {\phi_1}_ndx =o(1).\eeq

By Lemma \ref{aux}, using that $|Pw_1|=O(|\log\la_n|)$, we get
\beq\label{ludos2}\begin{aligned}\frac{2\rho}{\into
e^{W_1}}\into e^{W_1}{\phi_1}_nPw_1dx&=\into
e^{w_1}{\phi_1}_nPw_1dx+2\frac{\rho-4\pi}{\into e^z}\into
e^{z}{\phi_1}_nPw_1dx+o(1)\\ &= 8\langle {\tilde{\phi_1}}_n,
Pw_1({\delta_1}_n y)\rangle_{{\mathrm L}_2}
+2\frac{\rho-4\pi}{\into e^z}\into e^{z}{\phi_1}_nPw_1dx+o(1).
\end{aligned}
\eeq
Observe that by \eqref{pro-exp}
\beq\label{conve1} Pw_1\to -2\log |x|^2+8\pi H(x,0)\;\;\hbox{ in }L^q(\Omega)\quad \forall q\geq1.\eeq Moreover
$$Pw_1({\delta_1}_n y)=-2\log (1+|y|^2)+8\pi H({\delta_1}_n y, 0)-4\log{\delta_1}_n+O(\la_n),$$ by which  \beq\label{conve2}Pw_1({\delta_1}_n y)+4\log{\delta_1}_n\to -2\log (1+|y|^2)+8\pi H(0, 0)\;\hbox{ in }{\mathrm L}_2(\R^2).\eeq   Using these convergences into \eqref{ludos2}, and recalling Step 3, we obtain
\beq\label{ludo2}\begin{aligned}\frac{2\rho}{\into e^{W_1}}\into e^{W_1}{\phi_1}_nPw_1dx&=-16\gamma_1\Big\langle \frac{1-|y|^2}{1+|y|^2},\log (1+|y|^2) \Big\rangle_{{\mathrm L}_2} \\ &\;\;\;\;+64\pi H(0,0)\gamma_1\Big\langle  \frac{1-|y|^2}{1+|y|^2}, 1\Big\rangle_{{\mathrm L}_2}-32\log{\delta_1}_n\langle {\tilde{\phi_1}}_n, 1\rangle_{{\mathrm L}_2} +o(1)\\ &=8\pi\gamma_1-32\log{\delta_1}_n\langle {\tilde{\phi_1}}_n, 1\rangle_{{\mathrm L}_2}  +o(1)\end{aligned}
\eeq
by \eqref{comput1} and \eqref{comput2}.

Proceeding similarly as in \eqref{ludos2} with $1$ in the place of ${\phi_1}_n$ we deduce
\beq\label{ludo444}\begin{aligned}\frac{2\rho}{\into e^{W_1}}\into e^{W_1}Pw_1dx&=
8\langle 1, Pw_1({\delta_1}_n y)\rangle_{{\mathrm L}_2} +2\frac{\rho-4\pi}{\into e^z}\into e^{z}Pw_1dx+o(1)
\end{aligned}
\eeq
and using \eqref{conve1}-\eqref{conve2}

\beq\label{ludo3}\begin{aligned}\frac{2\rho}{\into
e^{W_1}}\into
e^{W_1}Pw_1dx&=-32\log{{\delta_1}_n}\|1\|_{{\mathrm L}_2}^2
-16\langle 1, \log(1+|y|^2)\rangle_{{\mathrm L}_2} +64\pi
H(0,0)\|1\|_{{\mathrm L}_2}^2\\ &\;\;\;\; + O(1)\\  &=-32\pi
\log{{\delta_1}_n} +O(1)
 \end{aligned}
\eeq by $\|1\|_{{\mathrm L}_2}^2=\pi$. Combining \eqref{sigma6}
with \eqref{ludo3} and taking into account \eqref{commento1} we
have \beq\label{ludo5}\begin{aligned}\frac{2\rho}{(\into
e^{W_1})^2}\into e^{W_1}Pw_1dx\into e^{W_1}{\phi_1}_n
dx&=-\frac{128}{\rho}\pi\log{\delta_1}_n\langle
{\tilde{\phi_1}_n}, 1\rangle_{{\mathrm L}_2} \\
&\;\;\;\;-\frac{32}{\rho}\pi\log{\delta_1}_n\frac{\rho-4\pi}{\into
e^z}\into e^{z} {\phi_1}_ndx+o(1).
\end{aligned}\eeq

Next we use Lemma \ref{aux} and  that $|Pw_1|=O(|\log\la_n|)$ and
we get \beq\label{ludos5}\begin{aligned}2\la_n \into
e^{W_2}{\phi_2}_n Pw_1dx  &=\into |x|^2e^{w_2}{\phi_2}_n
Pw_1dx+o(1)
\\ &=32\langle  {\tilde{\phi_2}}_n, Pw_1({\delta_2}_n y)\rangle_{{\mathrm L}_4}+o(1).
\end{aligned}\eeq
Observe that by \eqref{pro-exp}
$$\begin{aligned}Pw_1({\delta_2}_n y)&=-2\log ({\delta_1}_n^2+{\delta_2}_n^2 |y|^2  )  +8\pi H({\delta_2}_n y, 0)+O(\la)\\ &=-4\log {\delta_2}_n-2\log \Big(\Big(\frac{{\delta_1}_n}{{\delta_2}_n}\Big)^2+ |y|^2\Big)+8\pi H({\delta_2}_n y, 0)+O(\la_n)\end{aligned}$$
by which $$Pw_1({\delta_2}_n y)+4\log {\delta_2}_n\to -4\log |y|+8\pi H(0, 0)\;\;\hbox{ in }{\mathrm L}_4(\R^2).$$
By inserting this  convergence in \eqref{ludos5} we deduce
\beq\label{ludo4}\begin{aligned}2\la_n \into e^{W_2}{\phi_2}_n Pw_1dx& =-128\log {\delta_2}_n\langle  \tilde{\phi_2}_n , 1\rangle_{{\mathrm L}_4}
-128\gamma_2\Big\langle \frac{1-|y|^4}{1+|y|^4}, \log|y|\Big\rangle_{{\mathrm L}_4}\\ &\;\;\;\;+256\pi H(0,0)\gamma_2\Big\langle\frac{1-|y|^4}{1+|y|^4} ,  1\Big\rangle_{{\mathrm L}_4}+o(1)\\ &=4\pi\gamma_2+o(1)\end{aligned}\eeq
where we have used \eqref{comput1}, \eqref{comput3} and Step 3.

Finally, since $|Pw_1|=O(|\log\la_n|)$, we have $\into |\nabla
Pw_1|^2=\into e^{w_1}Pw_1=O(|\log\la_n|\into
e^{w_1})=O(|\log\la_n|)$ and so, by \eqref{inv3},
\beq\label{ludo6}\into|\nabla
{h_1}_n \nabla Pw_1|\leq\|{h_1}_n\|
\|Pw_1\|= o(1). \eeq

Passing to the limit in \eqref{1.3} and using \eqref{ludo1},
\eqref{ludo2},  \eqref{ludo5}, \eqref{ludo4} and
\eqref{ludo6}

$$-8\pi\gamma_1+\frac{32\rho-128\pi}{\rho}\log{\delta_1}_n \langle {\tilde{\phi_1}_n}, 1\rangle_{{\mathrm L}_2}-\frac{32}{\rho}\pi \log{\delta_1}_n\frac{\rho-4\pi}{\into e^z}\into e^{z} {\phi_1}_n+2\pi\gamma_2=o(1)$$ and then, by Step 3,
\beq\label{step4one}4\gamma_1-\gamma_2=o(1).\eeq

\bigskip

Next, we   multiply the second equation in \eqref{lla} by $   Pw_2
$, we integrate over $\Omega$ and we get
\begin{align}\label{1.4}& \into \nabla{\phi_2}_n\nabla Pw_2dx +\rho \[{\into e^{W_1}{\phi_1}_nPw_2dx \over \into e^{W_1   }dx}-
{\into e^{W_1}Pw_2dx  \into
e^{W_1} {\phi_1}_ndx\over\( \into e^{W_1}dx
\)^2}  \]\nonumber \\ & -2 \la_n \into
e^{W_2}{\phi_2}_n Pw_2dx     =-\into
\nabla{h_1}_n\nabla Pw_1dx.\end{align}

Again, we estimate each of the terms above. By \eqref{commento2},

\beq\label{ludo11}\into \nabla{\phi_2}_n\nabla
Pw_2dx=\into |x|^2e^{w_2}{\phi_2}_ndx=o(1).\eeq

By Lemma \ref{aux}, using that $|Pw_2|=O(|\log\la_n|)$, we get
\beq\label{ludoss22}\begin{aligned}2\la_n\into
e^{W_2}{\phi_2}_nPw_2dx&=\into|x|^2
e^{w_2}{\phi_2}_nPw_2dx+o(1)= 32\langle {\tilde{\phi_2}}_n,
Pw_2({\delta_2}_n y)\rangle_{{\mathrm L}_4} +o(1) .
\end{aligned}
\eeq
Observe that by \eqref{pro-exp}
$$Pw_2({\delta_2}_n y)=-2\log (1+|y|^4)+16\pi H({\delta_2}_n y, 0)-8\log{\delta_2}_n+O(\la_n),$$ by which  \beq\label{conve22}Pw_2({\delta_2}_n y)+8\log{\delta_2}_n\to -2\log (1+|y|^4)+16\pi H(0, 0)\;\;\hbox{ in }{\mathrm L}_4(\R^2).\eeq   Using these convergences into \eqref{ludoss22}, and recalling Step 3, we obtain
\beq\label{ludo22}\begin{aligned}2\la_n\into e^{W_2}{\phi_2}_nPw_2dx&=-64\gamma_2\Big\langle \frac{1-|y|^4}{1+|y|^4},\log (1+|y|^4) \Big\rangle_{{\mathrm L}_4} +512\pi H(0,0)\gamma_2\Big\langle  \frac{1-|y|^4}{1+|y|^4}, 1\Big\rangle_{{\mathrm L}_4}\\ &\;\;\;\,-256\log{\delta_2}_n\langle {\tilde{\phi_2}}_n, 1\rangle_{{\mathrm L}_4} +o(1)\\ &=16\pi\gamma_2+o(1)
\end{aligned}
\eeq by \eqref{comput1} and \eqref{comput2}. Again by Lemma \ref{aux},
taking into account that $|Pw_2|=O(|\log\la_n|)$, we get
\beq\label{ludos22}\begin{aligned}\frac{2\rho}{\into
e^{W_1}}\into e^{W_1}{\phi_1}_nPw_2dx&=\into
e^{w_1}{\phi_1}_nPw_2dx+2\frac{\rho-4\pi}{\into e^z}\into
e^{z}{\phi_1}_nPw_2dx+o(1)\\ &= 8\langle {\tilde{\phi_1}}_n,
Pw_2({\delta_1}_n y)\rangle_{{\mathrm L}_2}
+2\frac{\rho-4\pi}{\into e^z}\into e^{z}{\phi_1}_nPw_2dx+o(1)
\end{aligned}
\eeq
Again by \eqref{pro-exp}
$$Pw_2\to -8\log |y|+16\pi H(x,0)\;\;\hbox{ in }L^q(\Omega) \;\;\forall q\geq 1.$$
Moreover $$\begin{aligned}Pw_2({\delta_1}_n y)&=-2\log ({\delta_2}_n^4+{\delta_1}_n^4 |y|^4  )  +16\pi H({\delta_1}_n y, 0)+O(\la_n)\\ &=-8\log {\delta_2}_n-2\log \Big(1+\Big(\frac{{\delta_1}_n}{{\delta_2}_n}\Big)^4|y|^4\Big)+16\pi H({\delta_1}_n y, 0)+O(\la_n)\end{aligned}$$
by which $$Pw_2({\delta_1}_n y)+8\log {\delta_2}_n\to 16\pi H(0, 0)\;\;\hbox{ in }{\mathrm L}_2(\R^2).$$
By inserting these convergences into \eqref{ludos22}  we obtain

\beq\label{ludo33}\begin{aligned}\frac{2\rho}{\into e^{W_1}}\into e^{W_1}{\phi_1}_nPw_2dx&=-64\log {\delta_2}_n\langle {\tilde{\phi_1}_n}, 1\rangle_{{\mathrm L}_2}+   128H(0,0)\pi \gamma_1\Big\langle \frac{1-|y|^2}{1+|y|^2}, 1\Big\rangle_{{\mathrm L}_2}\\ &\;\;\;\;+2\frac{\rho-4\pi}{\into e^z}\into e^{z}{\phi_1}_n(-8\log |y|+16\pi H(x,0))dx+o(1)\\ &=-64\log {\delta_2}_n\langle {\tilde{\phi_1}_n}, 1\rangle_{{\mathrm L}_2}+o(1).
\end{aligned}
\eeq
Similarly, by replacing ${\phi_1}_n$ by $1$ in \eqref{ludos22},
$$\begin{aligned}\frac{2\rho}{\into e^{W_1}}\into e^{W_1}Pw_2dx&=
8\langle 1, Pw_2({\delta_1}_n y)\rangle_{{\mathrm L}_2} +2\frac{\rho-4\pi}{\into e^z}\into e^{z}Pw_2dx+o(1) \\ &=-64\log{\delta_2}_n\|1\|_{{\mathrm L}_2}^2+128\pi H(0,0)\|1\|_{{\mathrm L}_2}^2\\ &\;\;\;\;+2\frac{\rho-4\pi}{\into e^z}\into e^{z}(-8\log|y|+16\pi H(x,0))dx+o(1) \\ &=-64\pi\log{\delta_2}_n+O(1).
\end{aligned}
$$
Combining this with \eqref{sigma6} and taking into account
\eqref{commento1} we arrive at
\beq\label{ludo55}\begin{aligned}\frac{2\rho}{(\into
e^{W_1})^2}\into e^{W_1}Pw_2dx\into e^{W_1}{\phi_1}_n
dx&=-\frac{256}{\rho}\pi\log{\delta_2}_n\langle
\tilde{\phi_1}_n,1\rangle_{{\mathrm L}_2}\\
&\;\;\;\;-\frac{64}{\rho}\pi\log{\delta_2}_n
\frac{\rho-4\pi}{\into e^{z}}\into e^{z} {\phi_1}_n dx
+o(1)\end{aligned}\eeq and, similarly to \eqref{ludo6},
\beq\label{ludo66}\into\nabla {h_2}_n \nabla Pw_2dx=o(1).\eeq

Passing to the limit in \eqref{1.4} and using \eqref{ludo11},
\eqref{ludo22}, \eqref{ludo33},  \eqref{ludo55} and \eqref{ludo66}
we arrive at

$$\frac{-32\rho+128\pi}{\rho}\log{\delta_2}_n\langle \tilde{\phi_1}_n,1\rangle_{{\mathrm L}_2}+\frac{32}{\rho}\pi\log{\delta_2}_n \frac{\rho-4\pi}{\into e^{z}}\into e^{z} {\phi_1}_n dx-16\pi\gamma_2=o(1)$$
by which, using Step 3, $\gamma_2=0$ and, consequently, by \eqref{step4one}, $\gamma_1=0$.

\bigskip

\noindent\textit{Step 5. Conclusion.}

\bigskip

We will show that a contradiction arises. According to Step 2 and Step 4 we have

$$\tilde{\phi_1}_n\to 0\ \hbox{  weakly in $\mathrm{H}_{2} (\rr^2)$ and strongly in $\mathrm{L}_{2} (\rr^2) $, } $$
$$\tilde{\phi_2}_n \to 0\ \hbox{  weakly in $\mathrm{H}_{4} (\rr^2)$ and strongly in $\mathrm{L}_{4} (\rr^2) $,} $$
and
$${\phi_1}_n\to 0 \ \hbox{weakly in $H^1_0(\Omega)$ and strongly in $L^q(\Omega)$ for any $q\ge2.$}$$

By Lemma \ref{aux}  we get
\beq\label{conc1}\begin{aligned}\frac{2\rho}{\into e^{W_1}}\into e^{W_1}{\phi_1}_n^2dx&=\into e^{w_1}{\phi_1}_n^2dx+2\frac{\rho-4\pi}{\into e^z}\into e^{z}{\phi_1}_n^2dx+o(1)\\ &=
8\| {\tilde{\phi_1}}_n\|^2_{{\mathrm L}_2} +2\frac{\rho-4\pi}{\into e^z}\into e^{z}{\phi_1}_n^2dx+o(1) =o(1)
\end{aligned}
\eeq
and
\beq\label{conc3}\begin{aligned}2\la_n\into e^{W_2}{\phi_2}_n^2dx&=\into|x|^2 e^{w_2}{\phi_2}_n^2dx+o(1)=
32\| {\tilde{\phi_2}}_n\|^2_{{\mathrm L}_4} +o(1) =o(1).
\end{aligned}
\eeq
Moreover, recalling \eqref{commento1}  and \eqref{sigma6}, \beq\label{conch}\frac{1}{\into e^{W_1}}\into e^{W_1}{\phi_1}_ndx=o(1).\eeq Next, we   multiply the first and the second equations in
\eqref{lla} by $   {\phi_1}_n$, we integrate over $\Omega$ and,
using \eqref{conc1}  and \eqref{conch}, we deduce
$$ \into |\nabla{\phi_1}_n|^2dx+\la_n\into e^{W_2} {\phi_2}_n{\phi_1}_n dx =o(1),$$
$$ \into \nabla{\phi_2}_n\nabla{\phi_1}_ndx-2\la_n\into e^{W_2} {\phi_2}_n{\phi_1}_n dx =o(1),$$
respectively. Combining the above identities we obtain
\beq\label{fine1} 2\into |\nabla{\phi_1}_n|^2dx+ \into \nabla{\phi_2}_n\nabla{\phi_1}_ndx=o(1).\eeq
Similarly, we   multiply the first and the second equations in \eqref{lla} by $   {\phi_2}_n$, we integrate over $\Omega$ and, using   \eqref{conc3},
$$ \into \nabla{\phi_1}_n\nabla{\phi_2}_ndx -2\rho \[{\into e^{W_1}{\phi_1}_n{\phi_2}_ndx \over \into e^{W_1}dx}-
{\into e^{W_1}{\phi_2}_ndx  \into e^{W_1} {\phi_1}_ndx\over\( \into e^{W_1}dx \)^2}  \] =o(1),$$
$$ \into |\nabla{\phi_2}_n|^2dx +\rho \[{\into e^{W_1}{\phi_1}_n{\phi_2}_ndx \over \into e^{W_1}dx}-
{\into e^{W_1}{\phi_2}_ndx  \into
e^{W_1} {\phi_1}_ndx\over\( \into e^{W_1}dx
\)^2}  \]=o(1),$$ by which\beq\label{fine2} 2\into
|\nabla{\phi_2}_n|^2dx+ \into
\nabla{\phi_2}_n\nabla{\phi_1}_ndx=o(1).\eeq Summing up
\eqref{fine1} and \eqref{fine2},
$$\into |\nabla{\phi_1}_n|^2dx+\into |\nabla{\phi_2}_n|^2dx \leq 2
\into |\nabla{\phi_1}_n|^2dx+ 2 \into
|\nabla{\phi_2}_n|^2dx + 2\into
\nabla{\phi_2}_n\nabla{\phi_1}_ndx=o(1).$$ A contradiction arises
with \eqref{inv2}.
 \end{proof}

\section{The contraction argument: proof of Theorem \ref{teo} and Theorem \ref{main}}

Once we have studied the solvability of the linearized problem,
we are in position to prove Theorem \ref{main}.

First let us rewrite problem \eqref{s} in a more convenient way.
For any $ p>1,$ let $$i^*_{p}:L^{p}(\Omega)\to H^1_0(\Omega)$$ be the
adjoint operator of the embedding
$i_{p}:H^1_0(\Omega)\hookrightarrow L^{p\over p-1 }(\Omega),$ i.e.
$u=i^*_{p}(v)$ if and only if $-\Delta u=v$ in $\Omega,$ $u=0$ on
$\partial\Omega.$ We point out that $i^*_{p}$ is a continuous
mapping, namely
\begin{equation}
\label{isp} \|i^*_{p}(v)\| \le c_{p} \|v\|_{p}, \ \hbox{for any} \ v\in L^{p}(\Omega),
\end{equation}
for some constant $c_{p}$ which depends on $\Omega$ and $p.$
Then, setting  ${
u}:=(u_1,u_2)$ and $   {i^*_{p}}({
u}):=\(i^*_{p}(u_1),i^*_{p}(u_2)\)$, problem \eqref{s} is equivalent to
\begin{equation}\label{ps}
 u=   {i^*_{p}}\({     F} ({    u})\) \end{equation}
where
$$
{     F} ({    u}):=\(2 \rho g(u_1)-\lambda f(u_2),  2\lambda
f(u_2) -\rho  g(u_1) \)$$ and
$$
  f (u_2):=e^{u_2}\;\;  \hbox{and}\ \;\; g(u_1):=  \frac{e^{u_1}}{\int\limits_{\Omega}e^{u_1}dx}.
$$
Next we  denote by $      L: H^1_0(\Omega) \times H^1_0(\Omega) \to H^1_0(\Omega) \times H^1_0(\Omega)$ the
linear operator defined by
$$
      L(     \phi):=     {i^*_{p}}\(      F'(      W_\la)      \phi \) - \phi, \quad \phi=(\phi_1,\phi_2),
$$
where
$$
      F'(W_\la)(\phi)=   \( \begin{aligned} 2\rho \[{e^{{W_1}_\la}\phi_1\over \into e^{{W_1}_\la}}-
{e^{{W_1}_\la} \into e^{{W_1}_\la} \phi_1dx\over\( \into e^{{W_1}_\la} \)^2}  \] - \la e^{{W_2}_\la}\phi_2
 \\
    -\rho \[{e^{{W_1}_\la}\phi_1\over \into e^{{W_1}_\la}}-
{e^{{W_1}_\la} \into e^{{W_1}_\la} \phi_1dx\over\(
\into e^{{W_1}_\la} \)^2}  \] + 2\la e^{{W_2}_\la}\phi_2
\end{aligned}\).$$
Notice that problem \eqref{lla} reduces to
\beq\label{fre} L[\phi]=h, \qquad \phi, h\in {\cal H}_k\times {\cal H}_k.\eeq
As a consequence of Proposition \ref{inv} we derive the invertibility of $L$.

\begin{prop}\label{ex} For any $p\in (1,2) $ there exist $\lambda_0>0$ and $C>0$ such that for any $\la \in(0, \la_0)$
and for any $h=(h_1,h_2)\in {\cal H}_k\times {\cal H}_k$  there is a unique solution $ \phi=(\phi_1,\phi_2)$ to the problem \eqref{fre}. In
particular, $L$ is invertible; moreover, $$\| L^{-1} \| \leq C
|\log \lambda |.$$

\end{prop}
\begin{proof}  Observe that the operator $\phi\mapsto i^*_{p}(F'(W_\la) \phi)$ is a compact operator in ${\cal H}_k \times {\cal H}_k$.
   Let us consider the case $h=0$, and take $\phi\in {\cal H}_k\times {\cal H}_k$ with $L[\phi]=0$. In other words,  $\phi$ solves
the system \eqref{lla} with $h_1=h_2=0$. Proposition \ref{inv} implies $\phi\equiv 0$. Then, Fredholm's alternative implies the existence and uniqueness result.

Once we have existence, the norm estimate follows directly from Proposition \ref{inv}.
\end{proof}

\bigskip

\noindent{\bf The nonlinear problem.} Recall that we are interested in
finding a solution $u$ of \eqref{ps} with $u=W_\la+ \phi$, for some
small $\phi \in {\cal H}_k \times {\cal H}_k$. In what follows we denote by
 $N: H^1_0(\Omega) \times H^1_0(\Omega) \to H^1_0(\Omega) \times H^1_0(\Omega) $ the nonlinear operator
$$
      N(\phi):={i^*_p}\(      F (      W_\la+     \phi)-      F ( W_\la)-      F' (      W_\la)     \phi \).
$$
Therefore, problem \eqref{ps} turns out to be equivalent to the
problem \beq\label{goal}   N(\phi) + L(     \phi)=\tilde{R}_\la, \quad \phi\in {\cal H}_k\times{\cal H}_k\eeq
with
$$\tilde{R}_\la=W_\la-i_{p}^*(F(W_\la))
.$$
Observe that $\tilde{R}_\la= i_{p}^*({R}_\la)$, where
${R}_\la$ is given in \eqref{rla}.

The following lemma will be of use in the following:

\begin{lemma}\label{aux2} For any $p\geq1$, $r_0>0$ and $\eta>0$ there exist $\lambda_0>0$
and $C>0$ such that, for any $\la \in(0, \la_0)$ and for any $u \in
H^1_0(\Omega) $ with $\|u\| \le r_0$,

\begin{enumerate}

\item[a)] $\|\la f'(W_2+u) \, v\|_{p}\leq
C\la^{\frac{1-p}{p}-\eta} \|v\| \quad \forall v \in
H^1_0(\Omega),$

\item[b)] $\|g'(W_1+u)\, v \|_{p} \leq C \la^{2\frac{1-p}{p}-\eta}  \|v\| \quad \forall v \in H^1_0(\Omega),$

\item[c)] $\|\la f''(W_2+u)\, v\, z\|_{p}\leq
C\la^{\frac{1-p}{p}-\eta} \|v\|\|z\|\quad \forall v, \
z\in H^1_0(\Omega),$

\item[d)] $\|g''(W_1+u)\, v\, z\|_{p}\leq C \la^{3\frac{1-p}{p}-\eta}
\|v\|\|z\|\quad \forall v, \ z\in H^1_0(\Omega).$
\end{enumerate} \end{lemma}

\begin{proof} The proof of this lemma is basically contained in \cite[Lemma 4.7]{noi}; however, we reproduce it here for the sake of completeness.
To start with, easy computations lead to the estimate:
\begin{equation} \label{hola} \| e^{{w_1}}\|_{p} =
O(\lambda^{\frac{1-p}{p}}), \ \ \| |x|^2 e^{{w_2}}\|_{p} =
O(\lambda^{\frac{1-p}{p}})\quad \forall p\geq 1. \end{equation}
By Lemma \ref{aux} we conclude that:
\begin{equation} \label{hola2} \| e^{W_1}\|_{p} =
O(\lambda^{\frac{1-p}{p}}), \ \ \| \lambda e^{W_2}\|_{p} =
O(\lambda^{\frac{1-p}{p}})\quad \forall p\geq 1. \end{equation}
We give the complete proof for inequalities c), d), the others
being easier. We point out that by  H\"older's inequality  with
${1\over q} +{1\over r}+{1\over s}+{1\over t} =1$,
$$\begin{aligned}
\left\| \la f''\(W_2+u\)vz\right\| _{p} &\le \left\| \la e^{W_2
} \right\| _{pq}\left\|   e^{ u} \right\| _{pr}
\left\|v \right\| _{ps}\left\| z \right\| _{pt}\nonumber\\
&\hbox{(we use the continuity of $H^1_0(\Omega)\hookrightarrow
L^p(\Omega)$)}\nonumber\\ &\le C \left\| \la e^{W_2 } \right\|
_{pq}\left\|   e^{ u} \right\| _{pr}\left\|v \right\|  \left\| z
\right\|\nonumber\\ &\hbox{(we use Lemma \ref{tmt})}\nonumber\\
&\le C \left\| \la e^{W_2 } \right\| _{pq} e^{ \frac
{pr}{16\pi}\|u \|^2} \left\|v  \right\|  \left\| z \right\|
\nonumber\\ &\le C \la^{1-pq\over pq}  e^{ \frac {pr}{16\pi}\|u
\|^2} \left\|v \right\|  \left\| z \right\|.
\end{aligned}$$
It suffices now to choose $q>1$ sufficiently small to obtain c). Moreover
\begin{align*}
g''(W_1+u)[v,z]=&{e^{W_1+u}\over\int\limits_{\Omega }e^{W_1+u}}vz-{e^{W_1+u}\over\(\int\limits_{\Omega }e^{W_1+u}\)^2}v\int\limits_{\Omega }e^{W_1+u}z-{e^{W_1+u}\over\(\int\limits_{\Omega }e^{W_1+u}\)^2}z\int\limits_{\Omega }e^{W_1+u}v\\
 &-{e^{W_1+u}\over\(\int\limits_{\Omega }e^{W_1+u}\)^2}\int\limits_{\Omega }e^{W_1+u}vz+2{e^{W_1+u}\over\(\int\limits_{\Omega }e^{W_1+u}\)^3}\int\limits_{\Omega }e^{W_1+u}v
 \int\limits_{\Omega  }e^{W_1+u}z.\end{align*}
We use H\"older's inequalities with $ {1\over \alpha }+{1\over
\beta }=1,$ $ {1\over a}+{1\over b }+{1\over d}=1,$ and ${1\over q
}+ {1\over r }+{1\over s }+{1\over t }=1$ and  $\alpha$, $a$, $q$
sufficiently close to 1: we obtain
$$\begin{aligned}
\left\|g''(W_1+u)[v,z]\right\|_{p}  \le & {\|e^{W_1}\|_{pq}\|e^{u }\|_{pr}\|v \|_{ps}\|z \|_{pt}\over\|e^{W_1+u }\|_{1}} + 2 {\|e^{W_1}\|^2_{pa}\|e^{ u }\|^2_{pb}\|v \|_{pd }\|z \|_{pd} \over\|e^{W_1+u }\|^2_{1}}\nonumber\\
&+{\|e^{W_1}\|_{p \alpha}\|e^{u }\|_{p \beta}
\|e^{W_1}\|_{pq}\|e^{u }\|_{pr}\|v \|_{ps}\|z \|_{pt}\over\|e^{W_1+u }\|^2_{1}}\nonumber\\
 &+2 {\|e^{W_1 }\|_{p \alpha}\|e^{u }\|_{p \beta}
 \|e^{W_1}\|^2_{pa}\|e^{ u }\|^2_{pb}\|v \|_{pd }\|z \|_{pd} \over\|e^{W_1+u }\|^3_{1}}\nonumber\\
 &\hbox{(we use the continuity of $H^1_0(\Omega)\hookrightarrow L^p(\Omega)$ and Lemma \ref{tmt})} \nonumber \\
 & \le C\la^{3\frac{1-p}{p}-\eta} e^{ c\|u\|^2} \|v\|\|z\|  .\end{aligned}$$

It is important to point out that
\begin{equation}
\|e^{W_1+u}  \|_{1}\ge c.
\end{equation}
Indeed, by Lemma \ref{aux}, it suffices to show that:

\begin{equation}
\| (e^{w_1} + e^z)e^{u}  \|_{1}\ge c.
\end{equation}
Clearly, $\int_{\Omega} e^{w_1}e^{u} \ge 0$. Moreover, 

\begin{align*}  \int_{\Omega} e^z e^{u}dx \geq c \int_{\Omega} e^{u} dx\geq c \int_{\Omega} e^{-|u|} dx\geq  c  |\Omega| e^{- \frac{1}{|\Omega|} \int_{\Omega} |u|} \geq
c  |\Omega| e^{-C}.
\end{align*}
In the above estimates the Jensen's inequality has been used.

\end{proof}

\begin{lemma}\label{B2} For any $\eta>0$, $r_0>0$ there exist $\lambda_0>0$ and $C>0$
such that for any $\la \in(0, \la_0)$ and for any $\phi,\ \psi
\in H^1_0(\Omega)\times H^1_0(\Omega)$ with $\|\phi\|, \|\psi\| \le r_0$
 \begin{equation}\label{B21}
\left \| {N} (\phi)\right\| \le C  \lambda^{-2\eta}\|\phi \|^2\end{equation}
  and
    \begin{equation}\label{B22}
\left\| {N} (\phi )-N(\psi)\right\| \le C \lambda^{ -2\eta}
   \|\phi -\psi\|(\|\phi \|+\|\psi \|).\end{equation}
  \end{lemma}
\begin{proof} Let us remark that (\ref{B21}) follows by choosing
$\psi =0 $ in (\ref{B22}) . Let us prove (\ref{B22}). First of
all, we point out that for any $p>1$
$$\left\| {N} (\phi )-N(\psi)\right\| \le c_{p}\left\| F (W_\la+\phi )-F (W_\la+\psi )-F' (W_\la)(\phi-\psi)\right\| _{p}.$$
We apply the mean value theorem (\cite[Theorem 1.8]{aprodi}) to
the map: $\varphi \mapsto F(\varphi + {W}_\la) - F'({W}_\la)\, \varphi \in
L^p(\Omega)\times L^p(\Omega)$, with $\varphi \in H_0^1(\Omega)\times H^1_0(\Omega)$. Then, there exists
$\theta \in (0,1)$ such that

$$ \| F (W_\la+ \phi )-F (W_\la+ \psi )-F'(W_\la) (\phi-\psi)\|_{p} \leq \|[F'(W_\la+\theta \phi+(1-\theta)\psi)-F'(W_\la)](\phi-\psi)\|_{p}.$$

We apply again the mean value theorem to the map $\varphi \mapsto
F'(\varphi + W_\la)(\phi- \psi)$; there exists $\sigma \in
(0,1)$ such that
\begin{align*}& \| \[F'(W_\la+\theta \phi+(1-\theta)\psi)-F'(W_\la)\](\phi-\psi)\|_{p} \\ &\leq  \| F''(W_\la+\sigma (\theta \phi+(1-\theta)\psi))(  \theta \phi+(1-\theta)\psi) (\phi-\psi)\|_{p}.\end{align*}
Taking into account that  $$F''(u)=(2\rho g''(u_1)-\la f''(u_2), 2\la f''(u_2)-\rho g''(u_1)),$$
Lemma \ref{aux2} allows us to conclude by choosing $p$ sufficiently close to 1.

\end{proof}

Now we are able to solve problem \eqref{goal}.
\begin{prop}\label{phi}
For any $\e\in (0,\frac14)$ there exists $\lambda_0>0$    such that
for any $\la \in(0, \la_0)$ there is a unique $\phi_\la \in {\cal H}_k
\times {\cal H}_k$ satisfying \eqref{goal} and
$$
\|     \phi_\la\|\leq \la^{\frac14-\e} .$$

\end{prop}
\begin{proof} Problem \eqref{goal} can be solved via a contraction mapping argument.
Indeed, in virtue of Proposition \ref{ex}, we can introduce the map
$$T(\phi):=L^{-1}\big(\tilde R_\la- N(     \phi)\big),\;\;\, \phi\in{\cal  H}_k \times {\cal H}_k.$$
 By \eqref{isp} and Lemma \ref{error}, recalling that $\tilde{R}_\la= i_{p}^*({R}_\la)$, we have:
$$\|\tilde{R}_\la\|=O(\la^{\frac14\frac{2-p}{p}})\quad \forall p>1,$$
 or, equivalently, \beq\label{estierror}\|\tilde{R}_\la\|=O(\la^{\frac14-\eta})\quad \forall \eta>0.\eeq
We claim that $T$ is a contraction map over the ball
\beq\label{ball}\left\{\phi\in {\cal H}_k \times {\cal H}_k\ :\ \|\phi\|\le
  \la^{\frac14-\e}\right\}\eeq provided  $\la$ is small enough. Indeed, using \eqref{estierror} and Lemma \ref{B2}, fixed $0<\eta<\min\{\e,\frac18-\frac{\e}{2}\}$, we have
$$\|T(\phi)\|\leq C|\log\la |(\la^{\frac 14 -\eta}+ \la^{-2\eta} \|\phi\|^2) < \la^{\frac14-\e}$$ and
$$\|T(\phi)-T(\psi)\|\leq C|\log\la| \la^{-2\eta } (\|\phi \|+\|\psi \|) \|\phi-\psi\|< \frac 12 \|\phi-\psi\|.$$
\end{proof}

\bigskip

\begin{proof}[Proof of Theorem \ref{main}.] It follows immediately
from Proposition \ref{phi}, since, as we have already observed, problem \eqref{s} is equivalent to  \eqref{goal}.

\end{proof}

\begin{proof}[Proof of Theorem \ref{teo}.] By \cite{blin, suzuki}
(see remark \ref{ok}), assumption \ref{H} is satisfied, and hence
Theorem \ref{main} provides us with a solution ${u}_\la=W_\la+\phi_\la$ of \eqref{eq:e-1} with
$$ \rho_1=\rho,\quad \rho_2={\rho_2}_\la =\la \into e^{{u_2}_\la} dx .$$
 Clearly, by \eqref{change} and \eqref{answ},

$$ {v_1}_\la = \frac{2 {u_1}_\la+{u_2}_\la}{3} = \frac 1 2 (P w_1 + z) + o(1), \ \ {v_2}_\la = \frac{2 {u_2}_\la+{u_1}_\la}{3} = \frac 1 2 P
w_2 + o(1)$$ in the $H^1$-sense and  the expansions of Theorem \ref{teo}
follow from \eqref{pro-exp}, recalling also \eqref{delta12}. Moreover, using H\"older's inequality with $\frac1a+\frac1b+\frac1c=1$ and \eqref{hola2}, 
$$\begin{aligned}\|e^{{u_1}_\la}- e^{{W_1}_\la}\|_1&=\|e^{{W_1}_\la+{\phi_1}_\la}-e^{{W_1}_\la }\|_{1}=\into e^{{W_1}_\la}|e^{{\phi_1}_\la}-1|dx\le \into e^{{W_1}_\la}e^{|{\phi_1}_\la|}|{\phi_1}_\la|dx\\ &\le C \|e^{{W_1}_\la}\|_{a} \|e^{{\phi_1}_\la}\|_{b} \|{\phi_1}_\la\|_{c} =o(1), \end{aligned}$$ if $a$ is chosen sufficiently close to 1.
Similarly,
%$$\|e^{{u_1}_\la}- e^{{W_1}_\la}\|_1=\|e^{{W_1}_\la+{\phi_1}_\la}- e^{{W_1}_\la}\|_1=o(1),$$
$$\|\la e^{{u_2}_\la}- \la e^{{W_2}_\la}\|_1=\|\la e^{{W_2}_\la+{\phi_2}_\la}- \la e^{{W_2}_\la}\|_1=o(1).$$
Then, by
Lemma \ref{aux}, for every $r>0$

$$\begin{aligned} \frac{{\rho_1}}{\int_{\Omega} e^{{u_1}_\la}}\int_{B(0,r)} e^{{u_1}_\la}dx&=
\frac{{\rho}}{\int_{\Omega} e^{{W_1}_\la}}\int_{B(0,r)}  e^{{W_1}_\la}dx+ o(1) \\
&= \frac 1 2 \int_{B(0,r)}  e^{w_1}dx + \frac{\rho-4\pi}{\into e^z} \int_{B(0,r)}
e^{z} dx+ o(1) \\ &\to 4 \pi + \frac{\rho-4\pi}{\into e^z } \int_{B(0,r)}
e^zdx
\end{aligned}$$ as $\la \to 0$. If we now make $r \to 0$, we obtain
that $\sigma_1 = 4 \pi$. Analogously, if we use Lemma \ref{aux},

\begin{align*} \lambda\int_{B(0,r)}  e^{{u_2}_\la} dx=\lambda
\int_{B(0,r)}  e^{{W_2}_\la}dx+ o(1) = \frac 1 2 \int_{B(0,r)}
|x|^2 e^{w_2} dx+ o(1) \to
8 \pi,
\end{align*}
as $\la \to 0$, by which ${\sigma_2}= 8\pi$. Moreover, in this case there is no global mass since, again by Lemma \ref{aux},
\begin{align*} {\rho_2}_\la= \lambda\int_{\Omega}  e^{{u_2}_\la}dx =
 \frac 1 2 \int_{\Omega} |x|^2
e^{w_2}dx + o(1) \to 8 \pi.
\end{align*}
This concludes the proof.
\end{proof}


\begin{thebibliography}{99}

\bibitem{aprodi}{A. Ambrosetti and G. Prodi, }{\textit{A primer of Nonlinear Analysis}, }{Cambridge University Press 1993.}

\bibitem{ao}{W. Ao and L. Wang, }{\textit{New concentration phenomena for $SU(3)$ Toda system}, }{J. Differential Equations 256 (2014), 1548--1580.}

%\bibitem{baraket}{S. Baraket, F. Pacard, }{Construction of singular limits for a semilinear elliptic equation
%in dimension 2, }{Calc. Var. PDE  6 (1998), 1-38.}

\bibitem{blin}{D. Bartolucci and C.-S. Lin, }{\textit{Existence and uniqueness for mean field equations
on multiply connected domains at the critical parameter}, }{Math.
Annalen 359 (2014), 1--44.}

\bibitem{bjmr}{L. Battaglia, A. Jevnikar, A. Malchiodi and D. Ruiz, }{\textit{A general existence result for the Toda system on compact surfaces}, }
{preprint arXiv: 1306.5404.}

\bibitem{bmancini}{L. Battaglia and G. Mancini, }{\textit{A note on compactness properties of the singular
Toda system,} }{preprint arXiv:1410.4991.}

\bibitem{breme}{H. Brezis and F. Merle F, }{\textit{Uniform estimates and blow-up
behavior for solutions of $-\Delta u =V(x) e^u$ in two
dimensions, }}{Commun. Partial Differ. Equations 16 (1991),
1223-1253.}

\bibitem{clin}{C.C Chen and C.S. Lin, }{\textit{Topological degree for a mean
field equation on Riemann surfaces,} }{Comm. Pure Appl. Math. 56
(2003), 1667--1727.}

\bibitem{noi}{T. D'Aprile, A. Pistoia and D. Ruiz, }{\textit{A continuum of solutions for the $SU(3)$ Toda system  exhibiting partial blow-up}, }{preprint  arXiv:1407.8407 }

\bibitem{dem}{M. Del Pino, P. Esposito and M. Musso, }{\textit{Nondegeneracy of entire solutions of
a singular Liouville equation}, }{Proc. Am. Math. Soc. 140 (2012), 581--588.}

%\bibitem{dkm} M. Del Pino, M. Kowalczyk,  M. Musso, Singular limits in
%Liouville-type equations. {Calc. Var. Partial Differential Equations}, {24}, (2005), 47-81.

\bibitem{djadli}{Z. Djadli, }{\textit{Existence result for the mean field problem
on Riemann surfaces of all genus}}, Comm. Contemp. Math. 10 (2008),
 205--220.

\bibitem{dm}{Z. Djadli and A. Malchiodi, }{\textit{Existence of
conformal metrics with constant $Q$-curvature}, }{Annals of Math.,
168 (2008), 813--858.}

\bibitem{dunne}{G. Dunne, }{\textit{Self-dual Chern-Simons Theories}, }{Lecture
Notes in Physics, vol. 36, Berlin: Springer-Verlag, 1995.}

%\bibitem{egpistoia}{P. Esposito, M. Grossi and A. Pistoia, }{On the existence of blowing-up solutions
%for a mean field equation, }{Ann. Ist. H. Poincar\'{e} Anal. Non Lin\'{e}aire 22, (2005), 227-257.}

\bibitem{gpistoia}{M. Grossi and A. Pistoia, }{\textit{Multiple blow-up phenomena for the $\sinh$-Poisson equation}, }
{Arch. Ration. Mech. Anal. 209 (2013), 287--320.}

\bibitem{guest}{M. A. Guest, }{\textit{Harmonic maps, loops groups, and integrable systems,} }{London Mathematical
Society Student Texts, 38. Cambridge University Press, Cambridge,
1997.}

\bibitem{jlw}{J. Jost, C. S. Lin and G. Wang, }{\textit{Analytic aspects of the Toda system II.
Bubbling behavior and existence of solutions,} }{Comm. Pure Appl.
Math. 59 (2006), 526--558.}

\bibitem{jw}{J. Jost and G. Wang, }{\textit{Analytic aspects of the Toda system I. A Moser-Trudinger
inequality}, }{Comm. Pure Appl. Math. 54 (2001), 1289--1319.}

%\bibitem{jw2}{J. Jost and G. Wang, }{\textit{Classification of solutions of a Toda
%system in $\R^2$}, }{Int. Math. Res. Not., 2002 (2002), 277--290.}

%\bibitem{li}{Y.-Y. Li, }{On a singularly perturbed elliptic equation, }{Adv. Differential Equations 2 (1997), 955-980.}
%
%
%\bibitem{lisha}{Y.Y. Li and I. Shafrir, }{Blow-up analysis for solutions of $- \Delta u =
%V e^u$ in dimension two, }{Indiana Univ. Math. J. 43 (1994),
%1255-1270.}

\bibitem{lin-wei-bo}{C.S. Lin, J. Wei and W. Yang, }{\textit{Degree counting and shadow system for
$SU(3)$ Toda system: one bubbling}, }{preprint arXiv:1408.5802.}

\bibitem{lyan}{C.S. Lin and S. Yan, }{\textit{Fully bubbling solutions for the $SU(3)$ Toda system of mean field type on a torus}, }{preprint.}

\bibitem{lwzao-GAFA}{C.S. Lin, J.C. Wei and C. Zao, }{\textit{Sharp estimates for fully bubbling solutions of a SU(3) Toda system}, }{Geom. Funct.
Anal. 22 (2012), 1591--1635.}

\bibitem{cheikh}{A. Malchiodi and C. B. Ndiaye, }{\textit{Some existence results for the Toda system on closed surfaces, }}
{Atti Accad. Naz. Lincei Cl. Sci. Fis. Mat. Natur. Rend. Lincei
(9) Mat. Appl. 18 (2007),  391--412.}

%\bibitem{mr}{A. Malchiodi and D. Ruiz, }{ New improved Moser-Trudinger
%inequalities and singular Liouville equations on compact surfaces,
%}{GAFA 21 (2011), 1196-1217.}

\bibitem{mruiz}{A. Malchiodi and D. Ruiz, }{\textit{A variational analysis of the Toda system on compact surfaces}, }{Comm. Pure Appl. Math. 66 (2013), 332--371.}


%\bibitem{mruiz2}{A. Malchiodi and D. Ruiz, }{\textit{On the Leray-Schauder degree of the Toda System on compact surfaces}, }{Proc. Amer. Math. Soc., to appear.}

\bibitem{Moe} J. Moser, {\textit{A sharp form of an inequality by N.Trudinger,}}
 Indiana Univ. Math. J.  { 20} (1970/71), 1077--1092.

\bibitem{mpwei}{M. Musso, A. Pistoia and J. Wei, }{\textit{New blow-up phenomena for $SU(N+1)$ Toda system}, }{preprint arXiv:1402.3784.}

\bibitem{osuzuki}{H. Ohtsuka, T. Suzuki, }{\textit{Blow-up analysis for SU(3) Toda system}, }{J. Differential Equations 232 (2007), 419--440.}

\bibitem{suzuki}{T. Suzuki, }{\textit{Global analysis for a two-dimensional elliptic eigenvalue problem with the exponential
nonlinearly}, }{Ann. Inst. H. Poincar{\'e} Anal. Non Lin{\'e}aire 9 (1992),
367--398.}

\bibitem{tar}{G. Tarantello, }{\textit{Self-Dual Gauge Field Vortices:
An Analytical Approach}, }{PNLDE 72, Birkh\"auser Boston, Inc.,
Boston, MA, 2007.}

\bibitem{yys}{Y. Yang, }{\textit{Solitons in Field Theory and Nonlinear
Analysis,} }{Springer-Verlag, 2001.}


\bibitem{Tru}  N. S.   Trudinger, {\textit{On imbeddings into Orlicz spaces and some applications,}}
 J. Math. Mech.  {  17} (1967), 473--483.







\end{thebibliography}
\end{document}